\documentclass[archive,article,ij4uq]{ij4uq}      %  review version
\usepackage[hang]{footmisc}
\fancypagestyle{plain}{%
  \fancyhf{}
  \fancyhead[R]{\small }
  \fancyfoot[R]{\small\bf\thepage }
  \fancyfoot[L]{\fottitle}
  }
\newcommand{\RR}{I\!\!R} %real numbers
\newcommand{\Nat}{I\!\!N} %natural numbers
\newcommand{\Var}{\mathbb{V}}
\newcommand{\Mean}{\mathbb{E}}

\usepackage{amsfonts,amsthm}

\usepackage{tikz} 
\usetikzlibrary{matrix,chains,positioning,decorations.pathreplacing,arrows} 
\usetikzlibrary{positioning}
\usepackage{hyperref}

\usepackage{cleveref}
\usepackage{natbib}

\newtheorem{prop}[theorem]{Proposition}
\newtheorem{coro}[theorem]{Corollary}
\usepackage{multicol}

\begin{document}

\volume{} 
\title{A Bayesian neural network approach to Multi-fidelity surrogate modelling}
\titlehead{Muti fidelity surrogate model with GP-BNN}
\authorhead{B. Kerleguer, C. Cannamela, \& J. Garnier}
%For at least  authors with different addresses, use instead the following commands
\corrauthor[1,2]{Baptiste Kerleguer}
\author[1]{Claire Cannamela}
\author[2]{Josselin Garnier}
\corremail{baptiste.kerleguer@polytechnique.edu}
\corraddress{CEA, DAM, DIF, F-91297, Arpajon, France}
\address[1]{CEA, DAM, DIF, F-91297, Arpajon, France}
\address[2]{Centre de Math\'ematiques Appliqu\'ees, Ecole Polytechnique, Institut Polytechnique de Paris, 91128 Palaiseau Cedex, France}
% End information for at least  authors with different addresses
% For authors with the same post address,
%\corrauthor{First A. Author}
%\corremail{f.author@affiliation.com}
%\author{Second B. Author, Jr.}
%\address{Department of Chemistry and Courant, Institute of Mathematical Sciences, New York, NY 10012, USA}
% End commands for all authors with the same address

\dataO{mm/dd/yyyy}
\dataO{}
\dataF{mm/dd/yyyy}
\dataF{}

\abstract{
  This paper deals with surrogate modelling of a computer code output in a hierarchical multi-fidelity context, i.e., when the output can be evaluated at different levels of accuracy and computational cost.
Using observations of the output at low- and high-fidelity levels, we propose a method that combines Gaussian process (GP) regression and Bayesian neural network (BNN), in a method called GPBNN.
The low-fidelity output is treated as a single-fidelity code using classical GP regression. 
The high-fidelity output is approximated by a BNN that incorporates, in addition to the high-fidelity observations, well-chosen realisations of the low-fidelity output emulator.
The predictive uncertainty of the final surrogate model is then quantified by a complete characterisation of the uncertainties of the different models and their interaction.
GPBNN is compared with most of the multi-fidelity regression methods allowing to quantify the prediction uncertainty.
}


\maketitle

\section{Introduction}
\label{sec::Intro}

We consider the situation in which two levels of code that simulate the same system have different costs and accuracies.
This framework is called hierarchical multi-fidelity.
We would like to build a surrogate model of the most accurate and most costly code level, also called high-fidelity code. 
The underlying motivation is to carry out an uncertainty propagation study or a sensitivity analysis that require many calls and therefore, the substitution of the high-fidelity code by a surrogate model with quantified prediction uncertainty is necessary.
To build  the surrogate, a small number $N_H$ of high-fidelity code outputs and a large number $N_L$ of low-fidelity code outputs are given.
In some applications we may have $N_L \gg N_H$, but this article will focus on $N_L > N_H$ and the context of small data.
Then, the low-fidelity surrogate model uncertainty must be taken into account.

A well-known method to build a surrogate model with uncertainty quantification is Gaussian process (GP) regression, GP 1F in this paper.
This method has become popular in computer experiments \cite{williams2006gaussian,santner2003design} and now allows scaling up in the number of learning points \cite{rulliere2018nested}.
The emergence of multi-fidelity codes (codes that can be run at different levels of accuracy and cost) has motivated the introduction of new GP regression approaches.
The first one was the Gaussian process auto-regressive or AR(1) scheme proposed by \cite{kennedy2000predicting}.
The form of the AR(1) model expresses a simple and linear relationship between the codes and it follows from a Markov property \cite{o1998markov}.
This method is used in \cite{forrester2007multi} for optimisation.
This approach has been improved by \cite{le2014recursive} with the decoupling of the recursive estimation of the hyper-parameters of the different code levels.
In the following the recursive AR(1) model is called the AR(1) model because the results do not change and only the computation time is reduced.
The Deep GP method introduced in \cite{perdikaris2017nonlinear} makes it possible to adapt the approach to cases in which the relationships between the code levels are nonlinear.
An improvement has been further made by adding to the covariance function of the high-fidelity GP proposed by \cite{perdikaris2017nonlinear} a linear kernel in \cite{cutajar2019deep}.
Multi-fidelity GP regression has been used in several fields as illustrated in \cite{song2019general,pilania2017multi}.
Multi-fidelity polynomial chaos expension (MF-PCE) can also be exploited as in \cite{ng2012multifidelity} for multi-fidelity regression.
MF-PCE are mainly used for sensitivity analysis, see \cite{palarGlobalSensitivityAnalysis2018}, but show their limitations for surrogate modelling in terms of uncertainty quantification.

Recent improvements in the implementation of neural networks have motivated research on multi-fidelity neural networks \cite{NEURIPS2020_60e1deb0}.
In \cite{meng2020composite}, the authors combine a fully connected neural network (NN) and a linear system for the interactions between codes.
The low-fidelity surrogate model is built using a fully connected NN, see \cite{ZHANG2021113485} for a direct application.
In order to quantify the prediction deviations and to evaluate the reliability of the prediction the NNs have been improved to become Bayesian Neural Networks (BNN) \cite{mackay1992practical}.
The multi-fidelity model has been improved by using BNN for high-fidelity modelling in \cite{meng2020multi}.
In our article this method will be called MBK.
The method in \cite{meng2020multi} offers options for single-fidelity active learning and is more general for multi-fidelity modelling.
The disadvantages of methods using NN are the non-prediction of the model uncertainties and the difficult optimisation of the hyper-parameters in a small data context.
These disadvantages can be overcome by the use of BNNs.
The ability of BNN for uncertainty quantification is explained in \cite{kabir2018neural}.

The purpose of our paper is to present a method competing in terms of prediction with the methods of multi-fidelity regression : AR(1) model, DeepGP and MBK.
Our paper also aims to improve the quantification of uncertainties in the non-linear case for all these methods.
As not all the proposed methods give Gaussian processes as outputs, the uncertainties have to be quantified in an appropriate way through prediction intervals.
An approach, combining two methods, GP regression and BNN, is proposed in our article.
We will use a Gauss-Hermite quadrature based method to transfer the low-fidelity GP a posteriori law to the high-fidelity BNN.
We will compare the properties and results of our strategy with the ones of these methods.
For evaluation, we examine our surrogate model approach with two-level multi-fidelity benchmark functions and with a simulation example.
The results demonstrate that the method presented in our paper is easier to train and more accurate than any other multi-fidelity neural network-based method.
Moreover, the method is more flexible compared to other GP-based methods, and it gives reliable estimations of the predictive uncertainties.

In our paper we propose to split the multi-fidelity surrogate modelling problem into two regression problems.
The first problem is the single-fidelity regression for the low-fidelity code.
The second problem is the regression for the high-fidelity code knowing the predictions and the predictive uncertainties of the low-fidelity surrogate model.
GP regression allows prediction with quantified uncertainty for the low-fidelity code, which is important to minimise the predictive error and to quantify the predictive uncertainty of the surrogate model of the high-fidelity code, as we will see below.
As in \cite{meng2020multi}, we want to use a BNN for the regression knowing the low-fidelity prediction.
The contribution we propose is an original strategy to transfer the low-fidelity predictive uncertainties to the BNN.
For that we take well-chosen realisations of the predictor of the low-fidelity code by a quasi-Monte Carlo method based on Gauss-Hermite quadrature nodes, and we give them as inputs of the BNN in addition to the high-fidelity code inputs.
A predictor is obtained by a weighted average of the BNN outputs corresponding to the different realisations.
The predictive variance can also be assessed with the same sample.

The paper is organized as follows.
\Cref{sec::Method} presents different methods to build single-fidelity surrogate models.
The complete multi-fidelity method is presented in \cref{sec::Fullmeth}.
The specific interaction between GP regression and BNN is explained in \cref{sec::MCGPGH}.
\Cref{sec::Ex} shows numerical results.
Based on these results, the interest of the method is discussed in \cref{sec::Disc}.  

%------------------------------------------------

\section{Background: Regression with uncertainty quantification}
\label{sec::Method}
In this section classical surrogate modelling methods with uncertainty quantification are presented.
The GP regression method  is presented in \cref{ssec::GPreg}.
The BNN method is presented in \cref{ssec::BNN}.
Here we want to predict the scalar output $y=f (\mathbf{x})$, with $y\in \RR$, of a computer code with input $\mathbf{x} \in \RR^d $ from data set $(\mathbf{x}_i,y_i)_{i=1}^N$ with $y_i=f(\mathbf{x}_i)$.

\subsection{Gaussian process regression}
\label{ssec::GPreg}
GP regression can be used to emulate a computer code with uncertainty quantification \cite{williams2006gaussian}.
The prior output model as a function of the input $\mathbf{x}$  is a Gaussian process  $Y(\mathbf{x})$ with mean $\mu(\mathbf{x})$ and stationary covariance function $C(\mathbf{x},\mathbf{x}^\prime)$.
Consequently, the posterior distribution of the output $Y(\mathbf{x})$ given  $Y(\mathbf{x}_1)=y_1, \ldots , Y(\mathbf{x}_N)=y_N$ is Gaussian with mean:
\begin{equation}
\label{eq:meanpostL}
  \mu_\star(\mathbf{x}) = \mu(\mathbf{x}) + \mathbf{r}(\mathbf{x})^T {\bf C}^{-1} (\mathbf{y} - \boldsymbol{\mu}),
\end{equation}
and covariance :
\begin{equation}
\label{eq:varpostL}
  C_\star (\mathbf{x}, \mathbf{x^\prime}) = C(\mathbf{x}, \mathbf{x^\prime}) - \mathbf{r}(\mathbf{x})^T {\bf C}^{-1} \mathbf{r}(\mathbf{x^\prime}),
\end{equation}
with the vector $\mathbf{r}(\mathbf{x})=(C(\mathbf{x},\mathbf{x}_1), \ldots, C(\mathbf{x},\mathbf{x}_N))^T$, the matrix ${\bf C}$ defined by $C_{i,j} = C(\mathbf{x}_i,\mathbf{x}_j)$, the vector $\boldsymbol{\mu} = (\mu(\mathbf{x}_1), \ldots, \mu(\mathbf{x}_N))^T$ and the vector $\mathbf{y} = (y_1, \ldots, y_N)^T$.
The covariance function is chosen within a parametric family of kernels, whose parameters are fitted by maximizing the log marginal likelihood of the data, see \cite[Chapter~2.2]{williams2006gaussian}.
For practical applications the implementation of \cite[Algorithm~2.1]{williams2006gaussian} can be used.

  %\textbf{PART}  &\textbf{DESCRIPTION} \\
  %\hline \\
  %Dendrite         &Input terminal \\
  %Axon             &Output terminal \\
  %Soma             &Cell body (contains cell nucleus) \\
\subsection{Bayesian neural network}
\label{ssec::BNN}
Neural networks have been used to emulate unknown functions based on data \cite{CIGIZOGLU200663}, and in particular computer codes \cite{TRIPATHY2018565}. 
Our goal, however, is also to quantify the uncertainty of the emulation.
BNN makes it possible to quantify predictive uncertainties.
Below we present the BNN structure and the priors for the parameters. 
%Below, we first present the BNN structure, and we then  discuss the choice of the priors for the parameters. 

%\subsubsection{The structure}
We present a BNN with one hidden layer.
Let $N_l$ be the number of neurons in the hidden layer.
The output of the first layer is
\begin{equation}\label{eq::FLpred}
  \mathbf{y}_{1} = \Phi \left( \mathbf{w}_{1} \mathbf{x} + \mathbf{b}_{1} \right),
\end{equation}
with $\mathbf{x} \in \RR^{d}$ the input vector of the BNN, $\mathbf{b}_{1} \in \RR^{N_l}$ the bias vector, 
 $\mathbf{w}_1 \in \RR^{N_l \times d}$ the weight matrix and $\mathbf{y}_1 \in \RR^{N_l}$ the output of the hidden layer.
 The function $\Phi :\RR^{N_l} \to \RR^{N_l}$ is of the form $\Phi(\mathbf{b}) = (\phi(b_j))_{j=1}^{N_l}$, 
where the activation function $\phi$ can be hyperbolic tangent or ReLU for instance.
The second (and last) layer is fully linear:
\begin{equation}
  \label{eq::LLpred}
  \mathit{BNN}(\mathbf{x}) = \mathbf{w}_2^T \mathbf{y}_{1} + b_{2},
\end{equation}
with $\mathbf{w}_2 \in \RR^{N_n }$ the weight matrix, $b_2 \in \RR$ the bias vector and  $ \mathit{BNN}(\mathbf{x}) \in \RR$  the scalar output of the BNN at point $\mathbf{x}$.

We use a Bayesian framework similar to the one presented in \cite{jospin2020hands}.
Let $\boldsymbol{\theta}$ denote the parameter vector of the BNN, which is here
$  \boldsymbol{\theta}= \left( \mathbf{w}_i, \mathbf{b}_i\right)_{i=1,2}$.
The probability distribution function (pdf) of the output given $\mathbf{x}$ and $\boldsymbol{\theta}$ is 
\begin{eqnarray}
  p(y| \mathbf{x}, \boldsymbol{\theta}, \sigma) =
  \frac{1}{\sqrt{2\pi}\sigma} \exp\Big( - \frac{(y-  \mathit{BNN}_{\boldsymbol{\theta}}(\mathbf{x}))^2}{2 \sigma^2 } \Big),
  \label{eq::outputs}
\end{eqnarray}
where  $\sigma^2$ is  the variance of the random noise added to account for the fact that the neural network is an approximation.
$\mathit{BNN}_{\boldsymbol{\theta}}(\mathbf{x})$ is the output of the neural network with parameter $\boldsymbol{\theta}$ at point $\mathbf{x}$.

Here we choose a prior distribution for $\left(\boldsymbol{\theta},\sigma\right)$ that is classic in the field of BNN \cite[Part 5]{jospin2020hands}.
The prior laws of the parameters $\left( \mathbf{w}_i, \mathbf{b}_i\right)_{i=1,2}$ are:
\begin{eqnarray}
 \mathbf{w}_i \sim \mathcal{N}\left(\mathbf{0}, \sigma_{w_i}^2\mathbf{I} \right),\quad 
  \mathbf{b}_i \sim \mathcal{N}\left(\mathbf{0}, \sigma_{b_i}^2\mathbf{I} \right),   \quad i=1,2,
  \label{eq::priorTheta}
\end{eqnarray}
with  $\sigma_{{w,b}_{1,2}}$ the prior standard deviations.
The prior for $\sigma$ is the standard Gaussian $\mathcal{N}(0,1)$ (assuming the function $f$ has been normalized to be of order one).
All parameters are  assumed to be independent.
%The choice of priors in BNN is discussed in \cite{NEURIPS2020_c70341de}
%The authors use GP as prior of BNN.
%Anther way to see prior for BNN came from \cite{matsubara2020ridgelet}.
%It is the same as presented here for one hidden layer.

Applying Bayes' theorem, the posterior pdf of  $\left(\boldsymbol{\theta},\sigma\right)$ given the data ${\cal D}=(\mathbf{x}_i, y_i)_{i=1}^N$ is 
\begin{equation}
p(\boldsymbol{\theta},\sigma | {\cal D}) = \prod_{i=1}^N p(y_i|\mathbf{x}_i,\boldsymbol{\theta},\sigma ) p(\boldsymbol{\theta},\sigma)
\end{equation}
up to a multiplicative constant, where $p(\boldsymbol{\theta}, \sigma)$ is the prior distribution of $\boldsymbol{\theta}, \sigma$ described above.
The posterior distribution of the output at $\mathbf{x}$ has pdf
\begin{equation}
\label{eq:postp}
p(y | \mathbf{x} ,{\cal D}) = \iint p(y |\mathbf{x},\boldsymbol{\theta}, \sigma) p(\boldsymbol{\theta},\sigma|{\cal D}) d\boldsymbol{\theta}d\sigma ,
\end{equation}
and the two first moments are:
\begin{equation}
  \label{eq::baymeanvar}
  \Mean_{\text{post}}\left[Y\right] = \iint \mathit{BNN}_{\boldsymbol{\theta}}(\mathbf{x}) p(\boldsymbol{\theta}, \sigma |{\cal D}) d\boldsymbol{\theta} d\sigma,
\end{equation}
\vspace{-0.35cm}
\begin{equation}
  \label{eq::secondmombay}
  \Mean_{\text{post}}\left[Y^2\right] = \iint \big( \mathit{BNN}_{\boldsymbol{\theta}}(\mathbf{x})^2 + \sigma^2\big) p(\boldsymbol{\theta}, \sigma |{\cal D}) d\boldsymbol{\theta} d\sigma.
\end{equation}

Contrarily to GP regression, the prediction of a BNN cannot be expressed analytically as shown by (\ref{eq:postp}) but there exist efficient sampling methods.
To sample the posterior distribution of the BNN output, we need to sample the posterior distribution of $\left(\boldsymbol{\theta},\sigma\right)$.
In this paper the No-U-Turn Sampler (NUTS), which is a Hamiltonian Monte-Carlo (HMC) method, is used to sample the posterior distribution of $\left(\boldsymbol{\theta},\sigma\right)$ \cite{hoffman2014no}.
By \cref{eq::baymeanvar,eq::secondmombay}, the estimated mean  $\tilde{f}$ and variance $\tilde{V}$ of the output at point $\mathbf{x}$ are:
\begin{align}
  \label{eq::MeanBNN}
  \tilde{f}(\mathbf{x}) &= \frac{1}{N_v} \sum_{i=1}^{N_v} \mathit{BNN}_{\boldsymbol{\theta}_i}(\mathbf{x}),\\
  \label{eq::VarBNN}
  \tilde{V}(\mathbf{x})& = \frac{1}{N_v} \sum_{i=1}^{N_v} \left[\mathit{BNN}_{\boldsymbol{\theta}_i}(\mathbf{x})-\tilde{f}(\mathbf{x})\right]^2 + \frac{1}{N_v} \sum_{i=1}^{N_v}\sigma_i^2,
\end{align}
where the $(\boldsymbol{\theta}_i,\sigma_i)_{i=1}^{N_v}$ is the HMC sample of $\left(\boldsymbol{\theta},\sigma\right)$ with its posterior distribution.

\textcolor{black}{In this paragraph we study the sampling algorithm for the posterior law.
The Markov Chain Monte Carlo (MCMC) method aims at generating the terms of an ergodic Markov chain $\left( X_n \right)_{n\in\Nat}$ whose invariant measure is the target law with density $p(\mathbf{x})$ which is known up to a multiplicative constant.
This Markov chain is specially constructed for this purpose, in the sense that its transition kernel is defined such that $p$ is its unique invariant probability.
There are several possible variations of this principle, such as the Metropolis-Hastings algorithm, see \cite[sec. 6.3.2]{metropolis1953} and \cite{robert2004}.}

\textcolor{black}{The Metropolis-Hastings algorithm is as follows.
  We give ourselves a starting point $\mathbf{x}_0$ and an exploration law, i.e. a family $\left( q(\mathbf{x}^\prime, \mathbf{x}) \right)_{\mathbf{x}\in\RR^d}$ of probability densities on $\RR^d$ parametrized by $\mathbf{x}^\prime \in \RR^d$ that are easily simulated.
  We assume that we have simulated the $n$-th term of the chain $X_n$.
  \begin{enumerate}
    \item We make a proposition $X^\prime_{n+1}$ drawn according to the density law $q(X_n, \cdot)$.
    \item We calculate the acceptance rate $a(X_n, X^\prime_{n+1}) = \min(1, \rho(X_n, X^\prime_{n+1}))$ with,
    \begin{equation}
      \label{eq::MHalgorho}
      \rho(\mathbf{x}_n, \mathbf{x}^\prime_{n+1}) = \frac{p(\mathbf{x}^\prime_{n+1})q(\mathbf{x}^\prime_{n+1}, \mathbf{x}_n)}{p(\mathbf{x}_{n})q(\mathbf{x}_{n}, \mathbf{x}^\prime_{n+1})}.
    \end{equation}
    \item We draw $U_{n+1} \sim \mathcal{U}(0, 1)$.
    \item We put:
    \begin{equation}
      X_{n+1} = \left\{
        \begin{array}{ll}
          X^\prime_{n+1} & \text{if} ~~ U_{n+1} \leq a(X_n, X^\prime_{n+1}) \\
          X_n & \text{if} ~~ U_{n+1} > a(X_n, X^\prime_{n+1}) \\
        \end{array}
        \right. 
    \end{equation}
  \end{enumerate}}

We note that we only need to know $p$ up to a multiplicative constant to be able to implement the algorithm according to \cref{eq::MHalgorho}.
In the case where $q$ is symmetrical $q(\mathbf{x}^\prime, \mathbf{x}) = q(\mathbf{x}, \mathbf{x}^\prime)$, the acceptance rate is simply $\min(1,\frac{p(X^\prime_{n+1})}{p(X_n)})$.
We have a symmetrical $q$ in particular when the exploration law is Gaussian: $(q(\mathbf{x}^\prime, \mathbf{x}))_{\mathbf{x}\in R^d }$ is the density of the law $\mathcal{N} (\mathbf{x}^\prime, \sigma^2I)$ with $\sigma^2$ to be calibrated in order to have an acceptance rate that is neither too high (which means that we do not explore enough), nor too low (which means that we reject the proposition too often because it is too far).
Usually we calibrate $\sigma^2$ during the simulation to observe a constant acceptance rate of the order of $\frac{1}{4}$, for more details see \cite{gelmanWeakConvergenceOptimal1997}.

\textcolor{black}{
The Hamiltonian Monte Carlo (HMC) algorithm is a special instance of  Metropolis-Hastings algorithm because the exploration law $q(\mathbf{x}^\prime,\mathbf{x})$ is determined by a Hamiltonian dynamic, in which the potential energy depends on the target density $p$.
This model is proposed in \cite{duaneHybridMonteCarlo1987}.
One algorithm that is efficient for HMC on the No-U-Turn Sampler (NUTS). It is used in this manuscript to sample the posterior distribution of $\left(\boldsymbol{\theta},\sigma\right)$, see \cite{hoffman2014no}.
}
%------------------------------------------------

\section{Combining GP regression and BNN}
\label{sec::Fullmeth}
From now on we consider a multi-fidelity framework with two levels of code, high $f_H$ and low $f_L$ fidelity, as in \cite{kennedy2000predicting}.
The input is $\mathbf{x} \in \RR^d$ and the outputs of both computer code levels $f_L(\mathbf{x})$ and $f_H(\mathbf{x})$ are scalar. 
We have access to $N_L$ low-fidelity points and $N_H$ high-fidelity points, with $N_H <N_L$.
In our article we focus on the small data framework where the low-fidelity code is not perfectly known.
Under such circumstances it remains uncertainty in the low-fidelity surrogate model.
If $N_H \ll N_L$ the situation would be different and we could assume that the low-fidelity emulator is perfect.

We therefore have two surrogate modelling tools: GP regression and BNN. To do multi-fidelity with non-linear interactions the standard methods use combinations of regression methods. With our two methods we can make four combinations: GP-GP also called DeepGP in \cite{perdikaris2017nonlinear,cutajar2019deep}, GP-BNN the method proposed in our paper, BNN-GP and BNN-BNN.
The Deep GP model will be compared to the proposed method in all examples of our paper. The BNN-BNN method would be extremely expensive and very close to the full NN methods by adding the predictive uncertainty. The logic behind our choice of GP-BNN over BNN-GP is as follows: if we assume that the low-fidelity code is simpler than the high-fidelity code, then it must be approximated by a simpler model.
GP regression is a surrogate model easier to obtain and it gives a Gaussian output distribution that can be sampled easily. Whereas BNN is more complex to construct and allows for more general output distributions to be emulated.
%Our attempts to construct BNN-GP by coupling them severely to a Monte Carlo sample were unsuccessful.

The code output to estimate is $f_H$ with the help of low- and high-fidelity points.
As the low-fidelity code $f_L$ is not completely known a regression method with uncertainty quantification, GP regression, is used,  to emulate it, as in \cref{ssec::GPreg}.
The output of the low-fidelity GP is then integrated into the input to a BNN, described in \cref{ssec::BNN}, to predict $f_H$.

The low-fidelity surrogate model is a GP $Y_L(\mathbf{x})$ built from $N_L$ low-fidelity data points $\left(\mathbf{x}_{L,i}, f_L(\mathbf{x}_{L,i}) \right) \in \RR^d \times \RR$.
The optimisation of the hyper-parameters of the GP is carried out in the construction of the surrogate model.
The GP is characterized by a predictive mean $\mu_L(\mathbf{x})$ and a predictive covariance $C_L(\mathbf{x},\mathbf{x}^\prime)$.

To connect the GP with the BNN the simplest way is to concatenate $\mathbf{x}$ and $\mu_L(\mathbf{x})$ (the best low-fidelity predictor) as input to the BNN.
However, this does not take into account the predictive uncertainty.
Consequently, we may want to add $C_L(\mathbf{x},\mathbf{x})$ or $\sqrt{C_L(\mathbf{x},\mathbf{x})}$ to the input vector of the BNN.
The idea is that the  BNN could learn from the low-fidelity GP more than from its predictive mean only, in order to give reliable predictions of the high-fidelity code with quantified uncertainties.
We will show that the idea is fruitful, and it can be pushed even further.

We have investigated three methods to combine the two surrogate models and to transfer the posterior distribution of the low-fidelity emulator to the high-fidelity one.
We demonstrate in \cref{sec::Ex} that the best solution is the so-called Gauss-Hermite method.

\section{Transfer methods}
\label{sec::MCGPGH}
The two learning sets are $\mathcal{D}^L = \{(\mathbf{x}_i^L, f_L(\mathbf{x}_i^L)), i=1,\ldots, N_L \}$ and $\mathcal{D}^H = \{(\mathbf{x}_i^H, f_H(\mathbf{x}_i^H)), i=1,\ldots, N_H \}$ with typically $N_H <N_L$ and we do not need to assume that the sets $\{x_i^L,i=1,\ldots,N_L\}$ and $\{x_i^H,i=1,\ldots,N_H\}$  are nested.

The low-fidelity model is emulated using GP regression, as a consequence the result is formulated as a posterior distribution given ${\cal D}^L$ that has the form of a Gaussian law.

\begin{prop}
    \label{po::postDistLow}
The posterior distribution  of $Y_L(\mathbf{x})$ knowing $\mathcal{D}^L$ is the Gaussian distribution  with mean $\mu_L(\mathbf{x})$ and variance $\sigma^2_L(\mathbf{x})$ of the form (\ref{eq:meanpostL}-\ref{eq:varpostL}).
We denote its pdf by $p(y_L|{\cal D}_L, \mathbf{x})$.
\end{prop}
\begin{proof}
The proof is given in \cite[chapter 2.2]{williams2006gaussian} (prediction with noise free observations).
\end{proof}
    
The posterior distribution of the high-fidelity code given the low-fidelity learning set $\mathcal{D}^L$ and the high-fidelity learning set $\mathcal{D}^H$ may have different forms depending on the input of the BNN.

\subsection{Mean-Standard deviation method and quantiles method}
In the Mean-Standard deviation method, called Mean-Std method, we give as input to the BNN the point $\mathbf{x}$ and the information usually available at the output of a GP regression, i.e. the predictive mean and standard deviation of the low-fidelity emulator at the point $\mathbf{x}$.

In this method, the input of the BNN whose output gives the prediction of the high-fidelity code at $\mathbf{x}$ is $\mathbf{x}^\text{BNN}=(\mathbf{x},\mu_L,\sigma_L)$.
The idea is that the BNN input consists of the input $\mathbf{x}$ of the code and of the mean and standard deviation of the posterior distribution of the low-fidelity emulator at $\mathbf{x}$.
The high-fidelity emulator is modelled as:
\begin{equation}
    \label{eq::GHsampling}
    Y_H(\mathbf{x}) = \mathit{BNN}_{\boldsymbol{\theta}}(\mathbf{x},\mu_L(\mathbf{x}),\sigma_L(\mathbf{x})) + \sigma\epsilon,
\end{equation}
with $\epsilon \sim \mathcal{N}(0,1)$.
We use the learning set ${\cal D}^H_{MS}=\big\{ \big( \mathbf{x}_i^H, \mu_L(\mathbf{x}_i^H), \sigma_L(\mathbf{x}_i^H), f_H(\mathbf{x}_i^H) \big), i=1,\ldots, N_H \}$  to train the BNN and get the posterior distribution of $(\boldmath{\theta},\sigma)$. Note that ${\cal D}^H_{MS}$ can be deduced from ${\cal D}^L$ and ${\cal D}^H$.

\begin{prop}
    \label{po::postDist}
 The posterior distribution of $Y_H(\mathbf{x})$ knowing $\mathcal{D}^L$ and ${\cal D}^H_{MS}$ has pdf
 \begin{equation}
 p\big(y_H | \mathbf{x}, \mathcal{D}^H_{MS}, \mathcal{D}^L \big) = \iint \frac{1}{\sqrt{2\pi} \sigma}
 \exp\Big( - \frac{(y_H-\mathit{BNN}_{\boldsymbol{\theta}}(\mathbf{x},\mu_L(\mathbf{x}),\sigma_L(\mathbf{x})))^2}{2 \sigma^2}\Big) p(\boldsymbol{\theta}, \sigma |{\cal D}^H_{MS}) d\sigma d\boldsymbol{\theta} ,
\end{equation}
 %        \begin{equation}
 %           \Mean\left[\varphi\left(Y_H(\mathbf{x})\right)| \mathcal{D}^H_{MS}, \mathcal{D}^L\right] = \iiint \frac{1}{\sqrt{2\pi}}e^{-\frac{\varepsilon^2}{2}} \varphi\left(\mathit{BNN}_{\boldsymbol{\theta}}(\mathbf{x},\mu_L(\mathbf{x}),\sigma_L(\mathbf{x})) + \sigma {\varepsilon} \right) p(\boldsymbol{\theta}, \sigma |{\cal D}^H_{MS})  d {\varepsilon}d\sigma d\boldsymbol{\theta} ,
%        \end{equation}
        with $ p(\boldsymbol{\theta}, \sigma |{\cal D}^H_{MS}) $ the posterior pdf of the hyper-parameters of the BNN.
\end{prop}

\begin{proof}
    Let $p(y_H|\boldsymbol{\theta}, \sigma, \mathbf{x}, \mathcal{D}^L)$ be the probability density function (pdf) of $Y_H(\mathbf{x})$ given by \cref{eq::GHsampling}.
    This pdf can be written:
    \begin{eqnarray*}
        p(y_H|\boldsymbol{\theta}, \sigma, \mathbf{x}, \mathcal{D}^L) & = & \frac{1}{\sqrt{2\pi} \sigma}\exp\Big( - \frac{(y_H-\mathit{BNN}_{\boldsymbol{\theta}}(\mathbf{x},\mu_L(\mathbf{x}),\sigma_L(\mathbf{x})))^2}{2 \sigma^2}\Big)
    \end{eqnarray*}
    The law of total probability gives that:
    \begin{eqnarray*}
        p\big(y_H | \mathbf{x}, \mathcal{D}^H_{MS}, \mathcal{D}^L \big) &= &\iint p(y_H|\boldsymbol{\theta}, \sigma, \mathbf{x}, \mathcal{D}^L) p(\boldsymbol{\theta}, \sigma |{\cal D}^H_{MS}) d\sigma d\boldsymbol{\theta}\\
        & = & \iint \frac{1}{\sqrt{2\pi} \sigma}
        \exp\Big( - \frac{(y_H-\mathit{BNN}_{\boldsymbol{\theta}}(\mathbf{x},\mu_L(\mathbf{x}),\sigma_L(\mathbf{x})))^2}{2 \sigma^2}\Big) p(\boldsymbol{\theta}, \sigma |{\cal D}^H_{MS}) d\sigma d\boldsymbol{\theta}.
    \end{eqnarray*} 
\end{proof}

\begin{coro}
    The mean and variance of the posterior distribution of $Y_H(\mathbf{x})$ knowing ${\cal D}^H_{MS}$ and $\mathcal{D}^L$ is:
    \begin{equation}\label{eq::meangenfull}
        \mu_H(\mathbf{x}) = \iint \mathit{BNN}_{\boldsymbol{\theta}} (\mathbf{x},\mu_L(\mathbf{x}),\sigma_L(\mathbf{x})) p(\boldsymbol{\theta}, \sigma |{\cal D}^H_{MS}) d\sigma d\boldsymbol{\theta},
    \end{equation}
    and
    \begin{equation}\label{eq::vargenfull}
        C_H(\mathbf{x}) = \iint \left(\mathit{BNN}_{\boldsymbol{\theta}}(\mathbf{x},\mu_L(\mathbf{x}),\sigma_L(\mathbf{x}))^2 + \sigma^2\right) p(\boldsymbol{\theta}, \sigma |{\cal D}^H_{MS}) d\sigma d\boldsymbol{\theta},
    \end{equation}
\end{coro}
\begin{proof}
    By the definition of the mean and variance we get:
    \begin{eqnarray*}
        \mu_H(\mathbf{x}) &= & \int y_H p\big(y_H | \mathbf{x}, \mathcal{D}^H_{MS}, \mathcal{D}^L \big) d y_H, \\
        C_H(\mathbf{x}) & = & \int (y_H - \mu_H(\mathbf{x}))^2 p\big(y_H | \mathbf{x}, \mathcal{D}^H_{MS}, \mathcal{D}^L \big) d y_H.
    \end{eqnarray*}
    We replace by the expressions given in \cref{po::postDist}.
    \begin{eqnarray*}
        \mu_H(\mathbf{x}) &= & \int y_H \iint \frac{1}{\sqrt{2\pi} \sigma}\exp\Big( - \frac{(y_H-\mathit{BNN}_{\boldsymbol{\theta}}(\mathbf{x},\mu_L(\mathbf{x}),\sigma_L(\mathbf{x})))^2}{2 \sigma^2}\Big) p(\boldsymbol{\theta}, \sigma |{\cal D}^H_{MS}) d\sigma d\boldsymbol{\theta} d y_H,\\
    \end{eqnarray*}
    which gives \cref{eq::meangenfull}.
    For the variance:
    \begin{eqnarray*}
        C_H(\mathbf{x}) &= & \int (y_H- \mu_H(\mathbf{x}))^2 \iint \frac{1}{\sqrt{2\pi} \sigma}\exp\Big( - \frac{(y_H-\mathit{BNN}_{\boldsymbol{\theta}}(\mathbf{x},\mu_L(\mathbf{x}),\sigma_L(\mathbf{x})))^2}{2 \sigma^2}\Big) p(\boldsymbol{\theta}, \sigma |{\cal D}^H_{MS}) d\sigma d\boldsymbol{\theta} d y_H,\\
    \end{eqnarray*}
    which gives \cref{eq::vargenfull}.
\end{proof}
\begin{coro}
    \label{po::MeanStandardEstimator}
    Given $\mathcal{D}^L$ and $\mathcal{D}^H_{MS}$, given a sample $(\theta_i,\sigma_i)_{i=1}^{N_v}$ of the posterior distribution of $(\boldmath{\theta},\sigma)$:
    \begin{eqnarray}
        \label{eq::meanEstGen}
        \tilde{\mu}_H(\mathbf{x}) = \frac{1}{N_v}\sum_{i=1}^{N_v}BNN_{\boldsymbol{\theta}_i}(\mathbf{x},\mu_L(\mathbf{x}),\sigma_L(\mathbf{x})),
    \end{eqnarray}
    and
    \begin{eqnarray}
        \tilde{C}_H(\mathbf{x}) = \frac{1}{N_v} \sum_{i=1}^{N_v} \left(  BNN_{\boldsymbol{\theta}_i}(\mathbf{x},\mu_L(\mathbf{x}),\sigma_L(\mathbf{x}))^2 + \sigma^2_i \right) - \tilde{\mu}_H(\mathbf{x})^2,
    \end{eqnarray}
    are consistent estimators of $\mu_H(\mathbf{x})$ and $C_H(\mathbf{x})$.
\end{coro}

The estimators need samples of the posterior distribution of $(\boldsymbol{\theta},\sigma)$.
By the HMC method (NUTS), we get a sample of the hyperparameters $\left(\boldsymbol{\theta}_j,\sigma_j\right)$  with the posterior distribution of the high-fidelity model.

\textcolor{black}{
It is possible to construct a similar method, called quantiles method that takes as input the quantiles instead of the standard deviation.
}
The Quantiles method consists of giving the mean and two quantiles of the low-fidelity GP emulator as input to the BNN.
Assuming we want to have the high-fidelity output uncertainty at level $\alpha \%$ we take the $\alpha/2\%$ and the $(1-\alpha/2)\%$ quantiles.
The expression of the BNN input vector is $\mathbf{x}^\text{BNN} = (\mathbf{x}, \mu_L, Q_{L,\alpha/2}, Q_{L,(1-\alpha/2)})$ and the high-fidelity learning set is \newline $\mathcal{D}^H_{Q} : \left\{\left((\mathbf{x}_i^H, \mu_L(\mathbf{x}_i^L), Q_{L,\alpha/2}(\mathbf{x}_i^H), Q_{L,(1-\alpha/2)}(\mathbf{x}_i^H)), f_H(\mathbf{x}_i^H)\right), i=1,\cdots, N_H\right\}$.
This method is very similar to the Mean-Std method and only the BNN input changes.
\textcolor{black}{
The estimators of the mean and variance have the same form.}

\subsection{Gauss-Hermite quadrature}
In this section we want to transfer the information about the posterior distribution of the low-fidelity emulator by a sampling method.
The GP posterior distribution is one-dimensional and Gaussian for each value of $\mathbf{x}$.
Therefore, a deterministic method for sampling is preferable in order to limit the number of calls to the BNN.
In the following this method is called GPBNN.
We propose to sample the Gaussian distribution by a quasi Monte Carlo method using $S$ Gauss-Hermite quadrature nodes \cite[Chapter 3]{gautschi2012numerical}.
This method has been chosen because it gives the best interpolation of a Gaussian distribution. 
The samples $\tilde{f}_{L,j}(\mathbf{x})$, with $j=1,\hdots, S$, of the GP posterior distribution are constructed using the roots $z_{S,j}$ of the physicists' version of the Hermite polynomials $  H_{S}(x) = (-1)^S e^{x^2} \partial_{x}^S e^{-x^2}$, $S\in \mathbb{N}$.
For each input $\mathbf{x}$ the GP posterior law has mean $\mu_L(\mathbf{x})$ and variance $C_L(\mathbf{x},\mathbf{x})$.
\textcolor{black}{We sample $S$ times the low-fidelity GP posterior law by using the Gauss-Hermite quadrature illustrated at \cref{fig::GaussHermite}.}
Therefore, the $j$th realisation in the Gauss-Hermite quadrature formula is:
\begin{equation}
  \label{eq::lowfirand}
  \tilde{f}_{L,j}(\mathbf{x}) = \mu_{L}( \mathbf{x}) + z_{S,j} \sqrt{2} \sqrt{C_L(\mathbf{x},\mathbf{x})},
\end{equation}
and the associated weight is $p_{S,j} = \frac{2^{S-1}S!}{S^2H_{S-1}^2(z_{S,j})}$,
for $j=1,\hdots,S$.
The learning set of the BNN is  $\mathcal{D}^H_{GH}: \left\{\left(  (\mathbf{x}_i^H, \mu_{L}(\mathbf{x}_i^H),\sigma_{L}(\mathbf{x}_i^H)), f_H(\mathbf{x}_i^H)\right), i=1,\cdots, N_H\right\}$.
\textcolor{black}{%In order to pass all the information from the GP to the BNN, a sampling is performed which allows taking into account the GP mean but also the covariance.
The GP covariance can be viewed in \cref{fig::GaussHermite} as a difference between the $S$ realisations of the GP posterior law.
The low-fidelity uncertainty prediction is taken into account in BNN input parameters.
}
\textcolor{black}{Thus, t}he high-fidelity emulator is modelled as:
\begin{equation}
  Y_H(\mathbf{x}) = \sum_{j=1}^S p_{S,j} \mathit{BNN}_{\boldsymbol{\theta}}(\mathbf{x},\tilde{f}_{L,j}(\mathbf{x})) +\sigma\epsilon,
\end{equation}
with $\epsilon\sim\mathcal{N}(0,1)$ and $\tilde{f}_{L,j}(\mathbf{x})$ is given by \cref{eq::lowfirand}.

\begin{figure}
  \begin{multicols}{2}
    \includegraphics[scale=0.5]{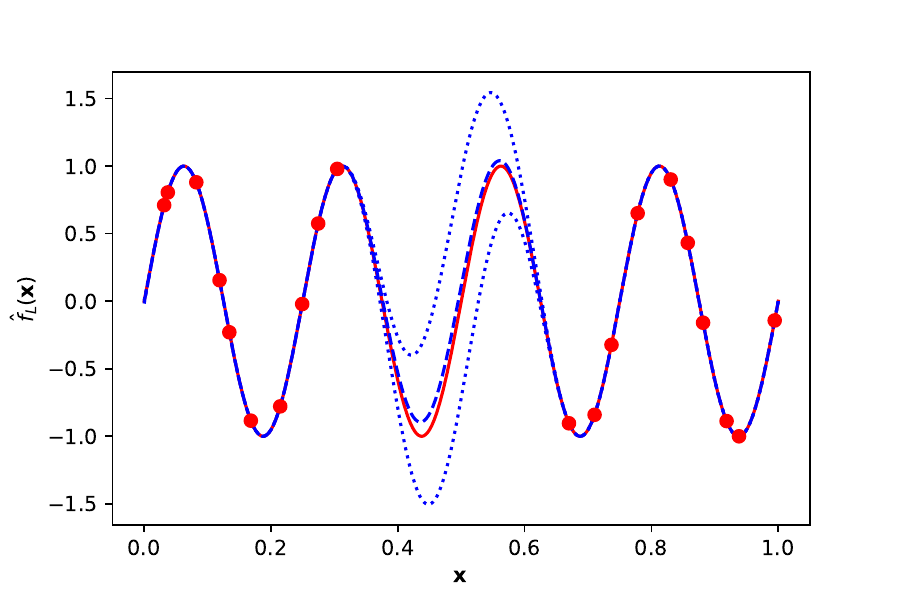} \par
    \includegraphics[scale=0.5]{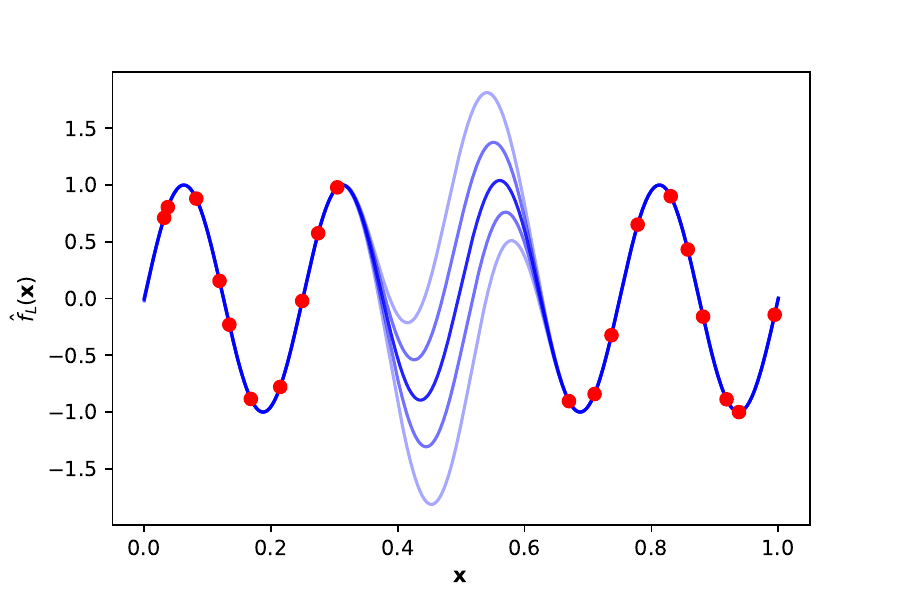} \par
  \end{multicols}
  %\begin{multicols}
    \begin{center}
      \begin{tabular}{c | c c c c c}
        $j$ & $1$ & $2$ & $3$ & $4$ & $5$ \\
        \hline
        $z_{S=5,j}$ & $-2.02$ & $-0.96$ & $ 0$ & $0.96$ & $2.02$ \\
        $p_{S=5},j$ & $0.020$ & $0.394$ & $0.945$ & $0.394$ & $0.020$\\
      \end{tabular} 
    \end{center}
  \caption{\label{fig::GaussHermite}\textcolor{black}{Illustration of a Gauss-Hermite quadrature for a GP.
  The left curve shows the posterior law of the GP.
  In red line we have the low-fidelity function to approach which we observe the red points.
  In dashed blue the mean of the GP and in dotted blue the $ \mathcal{I_{95\%}}$ the prediction interval.
  The right curve shows the GP with the realisations of the Gauss-Hermite quadrature for $S=5$.
  The intensity of the line is linear with the weight of the quadrature.
  The table gives the approximated values of the roots and weights for the Gauss-Hermite quadrature $S=5$.}}
\end{figure}
%The Gauss-Hermite method for \textcolor{red}{GP regression} is illustrated at \cref{fig::GaussHermite}.
% \textcolor{red}{%In order to pass all the information from the GP to the BNN, a sampling is performed which allows taking into account the GP mean but also the covariance.
% This covariance of the GP is shown in \cref{fig::GaussHermite} as a difference between the different realisations of the GP as a function of $\mathbf{x}$.
% The low-fidelity uncertainty is thus taken into account when predicting the BNN.
% }
%samples of the low-fidelity surrogate model with Gauss-Hermite.
%Its output is called $\mathit{BNN}_{\boldsymbol{\theta}}(\mathbf{x}, \tilde{f}_{L,i}(\mathbf{x}))$.

\begin{prop}
    \label{po::postDistGH}
    The posterior distribution of $Y_H(\mathbf{x})$ knowing $\mathcal{D}^H_{GH}$ and $\mathcal{D}^L$ is
        \begin{equation}
            p\left(y_H | \mathbf{x}, \mathcal{D}^H_{GH}, \mathcal{D}^L\right) = \iint \frac{1}{\sqrt{2\pi}\sigma}\exp\left(-\frac{1}{2\sigma^2}\left(y_H-\sum_{j=1}^S p_{S,j}\mathit{BNN}_{\boldsymbol{\theta}}(\mathbf{x},\tilde{f}_{L,j})\right)^2 \right) p(\boldsymbol{\theta}, \sigma |{\cal D}^H_{RS}) d\sigma d\boldsymbol{\theta} ,
        \end{equation}
        with $ p(\boldsymbol{\theta}, \sigma |{\cal D}^H_{GH}) $ the posterior distribution of the hyper-parameters of the BNN. $\tilde{f}_{L,j}$ is given by \cref{eq::lowfirand}.
\end{prop}
%\begin{coro}
%    \label{co::postDist}
%    \Cref{po::postDist} gives us that the posterior mean is:
%    \begin{eqnarray}
%        \label{eq::meanGeneral}
%        \mu_H(\mathbf{x}) = \iiint \mathit{BNN}_{\boldsymbol{\theta}}(\mathbf{x}^\text{BNN}) p(\boldsymbol{\theta}, \sigma %|{\cal D}^H) p(y_L(\mathbf{x})| \mathcal{D}^L) d\sigma d\boldsymbol{\theta} d y_L.
%    \end{eqnarray}
%    The posterior variance is:
%    \begin{eqnarray}
%        \label{eq::varGeneral}
%        C_H(\mathbf{x}) = \iiint \left(\mathit{BNN}_{\boldsymbol{\theta}}^2(\mathbf{x}^\text{BNN}) + \sigma^2\right) p(\boldsymbol{\theta}, \sigma |{\cal D}^H) p(y_L(\mathbf{x})| \mathcal{D}^L) d\sigma d\boldsymbol{\theta} d y_L - \mu_H^2(\mathbf{x}).
%    \end{eqnarray}
%\end{coro}
\begin{coro}
    \label{co::postMeanVarGH}
    The posterior mean of $Y_H(\mathbf{x})$ is:
    \begin{eqnarray}
        \label{eq::meanGeneral}
        \mu_H(\mathbf{x}) = \iint \left(  \sum_{j=1}^S p_{S,j}\mathit{BNN}_{\boldsymbol{\theta}}(\mathbf{x}, \tilde{f}_{L,j})\right) p(\boldsymbol{\theta}, \sigma |{\cal D}^H_{GH})  d\sigma d\boldsymbol{\theta} .
    \end{eqnarray}
    The posterior variance of $Y_H(\mathbf{x})$ is:
    \begin{eqnarray}
        \label{eq::varGeneral}
        C_H(\mathbf{x}) = \iint \left(\left(  \sum_{j=1}^S p_{S,j}\mathit{BNN}_{\boldsymbol{\theta}}(\mathbf{x}, \tilde{f}_{L,j})\right)^2 + \sigma^2\right) p(\boldsymbol{\theta}, \sigma |{\cal D}^H_{GH}) d\sigma d\boldsymbol{\theta}  - \mu_H^2(\mathbf{x}).
    \end{eqnarray}
    The expression of $ \tilde{f}_{L,j} $ is given by \cref{eq::lowfirand}.
\end{coro}

\begin{coro}
  Given $\mathcal{D}^L$ and $\mathcal{D}^H_{GH}$, given a sample $\left( \boldsymbol{\theta}_i, \sigma_i \right)_{i=1}^{N_v}$ of the posterior distribution of $\left( \boldsymbol{\theta}, \sigma \right)$,
  % We use the Gauss-Hermite quadrature of the low-fidelity GP, the high-fidelity learning set is $\mathcal{D}^H_{GH}$, with $ \tilde{f}_{L,i}(\mathbf{x})$ given at \cref{eq::lowfirand}.
  % The estimator of the high-fidelity mean of the output of the high-fidelity model is:
\begin{equation}
  \label{eq::prediction}
  \tilde{\mu}_H(\mathbf{x}) = \frac{1}{N_v}\sum_{i=1}^{N_v}\sum_{j=1}^S p_{S,j} \mathit{BNN}_{\boldsymbol{\theta}_i}(\mathbf{x}, \tilde{f}_{L,j}(\mathbf{x})),
\end{equation}
and
% and by \cref{co::postMeanVarGH} the estimator of the  variance is:
\begin{equation}
 \label{eq::prediction2}
  \begin{split}
    \tilde{C}_H(\mathbf{x}) = \frac{1}{N_v} \sum_{i=1}^{N_v} \left(\sum_{j=1}^S p_{S,j} \mathit{BNN}_{\boldsymbol{\theta}_i}(\mathbf{x},\tilde{f}_{L,j}(\mathbf{x}))\right)^2
    + \frac{1}{N_v} \sum_{i=1}^{N_v}\sigma_i^2 -  \tilde{\mu}_{H}^2(\mathbf{x}),
  \end{split}
  \end{equation}
  are consistent estimators of the mean and variance in \cref{co::postMeanVarGH}.
\end{coro}

%\begin{proof}
%We use the Gauss-Hermite quadrature for the sampling of $\tilde{f}_{H,j}(\mathbf{x})$.
%The estimation of the probability is then:
%\begin{eqnarray}
%  \int \mathit{BNN}_{\boldsymbol{\theta}}(\mathbf{x}^\text{BNN}) p(y_L(\mathbf{x})| \mathcal{D}^L) d y_L \approx \sum_{j=1}^S p_{S,j}\mathit{BNN}_{\boldsymbol{\theta}}(\mathbf{x}, \tilde{f}_{L,j}(\mathbf{x})).
%\end{eqnarray}
%The integration over $\boldsymbol{\theta}$ is done exactly as for the Random Sample method.
%Consequently, we have the estimators:
%\begin{eqnarray}
%  \tilde{f}_H(\mathbf{x}) = \frac{1}{N_v}\sum_{i=1}^{N_v}\sum_{j=1}^S p_{S,j} \mathit{BNN}_{\boldsymbol{\theta}_i}(\mathbf{x}, \tilde{f}_{L,j}(\mathbf{x})),
%\end{eqnarray}
%and
%\begin{eqnarray}
%  \tilde{C}_H(\mathbf{x}) = \frac{1}{N_v} \sum_{i=1}^{N_v} \sum_{j=1}^S\left( p_{S,j} \mathit{BNN}_{\boldsymbol{\theta}_i}(\mathbf{x},\tilde{f}_{L,j}(\mathbf{x}))^2 + p_{S,j}\sigma_i^2\right) -  \tilde{f}_{H}^2(\mathbf{x}) .
%\end{eqnarray}
%The term in $\sigma$ is independent of the sampling in $S$.
%Also, $\sum_{j=1}^S p_{S,i}=1$ then we can split the variance in the BNN part and $\sigma$ part.
%\end{proof}
%In \cref{sec::Ex} we study the values of $N_v$ .
% $S$ is sufficient to get proper estimations.
% When $N_v\rightarrow + \infty$, $ \tilde{\mu}_H(\mathbf{x})$ tends to 
% A FINIR 
%$\tilde{f}_H(\mathbf{x})$, given in \cref{eq::prediction} tends to the mean of the posterior law of the high-fidelity when $N_v$ and $S$ tend to $+\infty$. 

Note that this sampling method is different from the Quantiles method (even for $S=3$ and $\alpha\approx 0.110$).
Indeed, in the Quantiles method the input of the BNN contains the mean and quantiles of the low-fidelity code predictor whereas the Gauss-Hermite method is a sampling method in which the input of the BNN contains only one sample of the low-fidelity code predictor and the predictor is a weighted average of the BNN.
% The weights are different between the random samples and the Gauss-Hermite method.

The sample $\left( \boldsymbol{\theta}_i, \sigma_i \right)$ of the posterior distribution of $\left(\boldsymbol{\theta}, \sigma\right)$ can be obtained by the HMC method (NUTS).
% We get $N_v$ samples $\left(\boldsymbol{\theta}_i,\sigma_i\right)$ with $i=1,\hdots, N_v$.
The choice of $S$ is a trade-off between a large value that is computationally costly and a small value that does not propagate the uncertainty appropriately.
$S=2$ is the smallest admissible value regarding the information that should be transferred.
At first glance, a large value of $S$ could be expected to be the best choice in the point of view of the predictive accuracy.
However, a too large value of $S$ degrades the accuracy of the predictive mean estimation.
This is due to large variations of $\tilde{f}_{L,i}(\mathbf{x})$ for large values of $S$.
Interesting values turn out to be between $3$ and $10$ depending on the quality of the low-fidelity emulator, as discussed in section \ref{sec::Ex}.
\Cref{fig::SMuFi} is a graphical representation of the GPBNN method.

The value of $N_v$ is chosen large enough so that the estimators in \cref{eq::prediction,eq::prediction2} have converged.
As shown in the analysis of \cref{sec::Ex} $N_v=500$ is sufficient.
% For $N_v$ we try to take the highest possible value.
% However, the calculation time imposes a limit.
% We try to have the smallest value which in our case always converges.

We could expect the computational cost of the GPBNN method to be expensive.
% This is due to the fact that we combine Monte-Carlo methods.
The number of operations needed to compute an iteration of the HMC optimisation is proportional to $S\times N_v \times N_H$.
Because $S$ and $N_H$ are small in our context the computational cost of one realisation of the BNN is actually low.
Thus optimisation of the hyperparameters is feasible for $N_H \lesssim 100$.

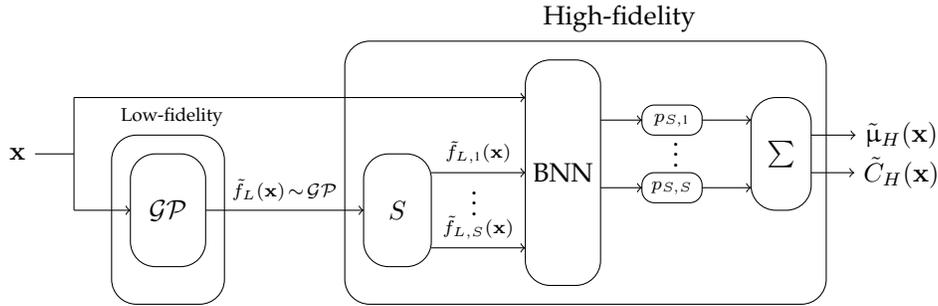
\begin{figure}
  \centering
  \begin{tikzpicture}
    % inputs
    \node at (1.25 , 1.75) (x) {$\mathbf{x}$};
    \node at (7.32 , 1.2) (dotsIn) {$\vdots$};
    \node at (10 , 1.85) (dotsOut) {$\vdots$};
    \node at (13, 2) (y) {$\tilde{\mu}_H(\mathbf{x})$};
    \node at (13, 1.5) (Cy) {$\tilde{C}_H(\mathbf{x})$};
    %rectangles
    \draw [color = black, rounded corners =2ex] (2.75,0.25) rectangle(3.75,1.75) node[pos=.5] (GP) {$\mathcal{GP}$};
    \draw [color = black, rounded corners =2ex] (8,0) rectangle(9,3) node[pos=.5] {BNN};
    \draw [color = black, rounded corners =2ex] (5.6,-0.25) rectangle(12,3.25) node[above, xshift=-2.75cm] {High-fidelity};
    \draw [color = black, rounded corners =2ex] (2.5,-0.25) rectangle(4,2) node[above, xshift=-0.7cm] {\tiny{Low-fidelity}};
    \draw [color = black, rounded corners =2ex] (11,1) rectangle(11.8,2.5) node[pos=.5] {$\sum$};
    \draw [color = black, rounded corners =2ex] (5.85,0.25) rectangle(6.75,1.75) node[pos=.5] (Sampling) {$S$};
    \draw [color = black, rounded corners =1ex] (9.55,2.4) rectangle(10.35,2) node[pos=.5] (weight1) {\tiny{$p_{S,1}$}};
    \draw [color = black, rounded corners =1ex] (9.55,1.5) rectangle(10.35,1.1) node[pos=.5] (weight2) {\tiny{$p_{S,S}$}};
    % arrows
    \draw[->] (2, 1) -- (2.75, 1); % arrow for GP input
    \draw (x) -- (2, 1.75);
    \draw (2, 1.75) -- (2, 2.5);
    \draw (2, 1.75) -- (2, 1);
    \draw[->] (2, 2.5) -- (8,2.5); % arrow for BNN input
    \draw[->] (3.75,1) -- node [auto] {\tiny{$ \tilde{f}_{L}(\mathbf{x})\! \sim\! \mathcal{GP} $}} (5.85,1); % arrow for interaction GB-BNN
    \draw[->] (9,1.3) --  (9.55,1.3); % arrow before p
    \draw[->] (9,2.2) --  (9.55,2.2); % arrow before p
    \draw[->] (10.35,1.3) --  (11,1.3); % arrow after p
    \draw[->] (10.35,2.2) --  (11,2.2); % arrow after p
    \draw[->] (6.75,1.5) -- node [auto] {\tiny{$\tilde{f}_{L,1}(\mathbf{x})$}}(8,1.5); % arrow after Gauss sampling
    \draw[->] (6.75,0.5) -- node [auto] {\tiny{$\tilde{f}_{L,S}(\mathbf{x})$}}(8,0.5);
    \draw[->] (11.8,2) -- (y);
    \draw[->] (11.8,1.5) -- (Cy);
    % arrow for Gaussian weights
    %\draw (6.5,0.25) -- (6.5,-0.25);
    %\draw (6.5,-0.25) --  (10,-0.25);
    %\draw[->] (10,-0.25) -- node [auto,near end] {$p_{S,i}$} (10,1);

  \end{tikzpicture}
  \caption{\label{fig::SMuFi}Schematic of the multi-fidelity Gauss-Hermite model. The input is a point $\mathbf{x}$. The output consists of a predictive mean $\tilde{\mu}_H(\mathbf{x})$ and a predictive variance $\tilde{C}_H(\mathbf{x})$.}
\end{figure}

\section{Experiments}
\label{sec::Ex}
In this section we present two analytic examples and a simulated one. 
The first one is a 1D function, and we consider that the low-fidelity function may be unknown in a certain subdomain.
The second one is a 2D function with noise.
Finally, we test the strategy on a more complex double pendulum system.
All the numerical experiments are carried out on a laptop (of 2017, dell precision with intel core i7) using only CPU and the running time never exceeds 2 hours.

\subsection{1D function approximation}
\label{sec::ex1D}
The low- and high-fidelity functions are:
\begin{align}
\label{eq::highex1Dquad}
  f_L(x) = \sin{8\pi x}, \quad \quad 
  f_H(x) = \big(x-\sqrt{2}\big) f_L^2(x),
\end{align}
with $x \in \left[0,1\right]$, where $f_H$ is the high-fidelity function (code) and $f_L$ the low-fidelity function (code).
\textcolor{black}{A graphical representation of these functions is given in \cref{fig:function}.}
These functions have been introduced in \cite{perdikaris2017nonlinear} and are well estimated with a DeepGP and a quadratic form of the covariance.
In this example we assume that we have access to a lot of low-fidelity data, $N_L = 100$, while the high-fidelity data is small, $N_H = 20$.
We also consider situations in which there is a segment $\bar{I} \subset \left[0,1\right]$ where we do not have access to $f_L(x)$.
The learning set for the high-fidelity code is obtained by partitioning $[0,1]$ into $N_{H}$ segments with equal lengths and then by choosing independently one point randomly on each segment with uniform distribution.
The learning set for the low-fidelity code is obtained by partitioning $[0,1] \backslash \bar{I}$ into $N_{L}$ segments with equal length and then by choosing independently one point randomly on each segment with uniform distribution.
% \textcolor{red}{
% In order to place in different configurations of learning we generate different $\bar{I}$.
% The first set is made to build a model with few low-fidelity uncertainties $\bar{I} = \emptyset$.
% In this case the low-fidelity model is accurate in predicting.
% }
The test set is composed of $N_{\rm T} = 1000$ independent points following a random uniform law on $[0,1]$.
\begin{figure}
  \centering
  \includegraphics[scale=0.5]{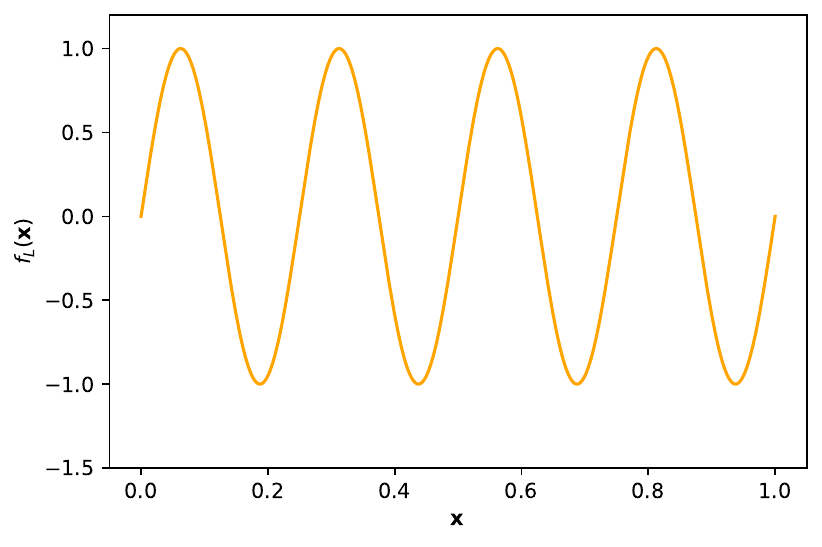}
  \includegraphics[scale=0.5]{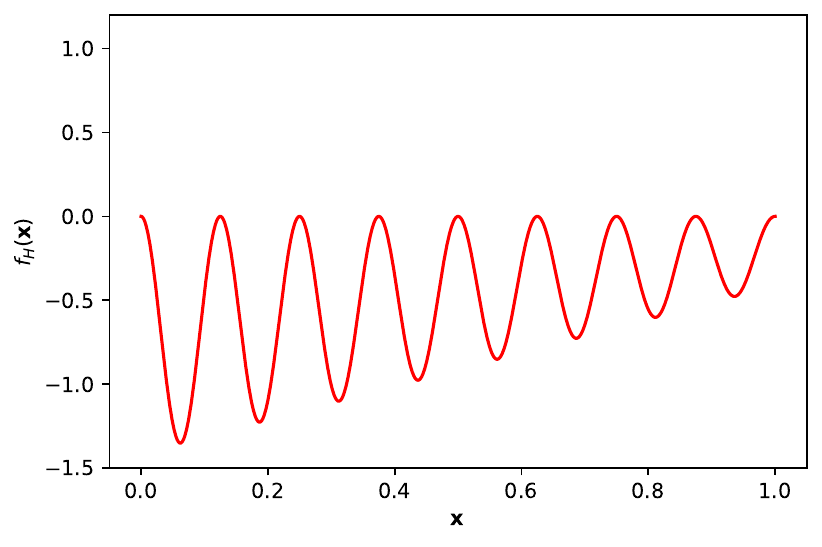}
  \caption{\label{fig:function}\textcolor{black}{Low-fidelity function on left and high-fidelity function on right. We can see that the range of variation is smaller for the high-fidelity function and its frequency is $2$ times higher than low-fidelity function.}}
\end{figure}

We denote by $\big(\mathbf{x}_{\rm T}^{(i)},f_H(\mathbf{x}_{\rm T}^{(i)})\big)_{i=1,\hdots,N_{\rm T}}$ the test set.
The error is evaluated by:
\begin{eqnarray}
  \label{eq::Q2formula}
  Q_{\rm T}^2 = 1 - \frac{\sum_{i=1}^{N_{\rm T}}\big[\tilde{\mu}_H(\mathbf{x}_{\rm T}^{(i)})-f_H(\mathbf{x}_{\rm T}^{(i)})\big]^2}{N_{\rm T}\Var_{\rm T}\left(f_H\right)},
\end{eqnarray}
with $\Var_{\rm T}(f_H) = \frac{1}{N_{\rm T}}\sum_{i=1}^{N_{\rm T}}\big[f_H(\mathbf{x}_{\rm T}^{(i)})- \frac{1}{N_{\rm T}} \sum_{j=1}^{N_{\rm T}} f_H(\mathbf{x}_{\rm T}^{(j)})\big]^2$.
A highly predictive model gives a $Q_{\rm T}^2$ close to $1$ while a less predictive model has a smaller $Q_{\rm T}^2$.
The coverage probability ${\rm CP}_\alpha$ is defined as the probability for the actual value of the function to be within the prediction interval with confidence level $\alpha$ of the surrogate model:
\begin{eqnarray}
{\rm CP}_\alpha = \frac{1}{N_{\rm T}}\sum_{i=1}^{N_{\rm T}} \mathbf{1}_{f_H(\mathbf{x}_{\rm T}^{(i)})\in \mathcal{I_{\alpha}}(\mathbf{x}_{\rm T}^{(i)})} ,
\end{eqnarray}
with $\mathbf{1}$ the indicator function and $ \mathcal{I_{\alpha}}(\mathbf{x})$ the prediction interval at point $\mathbf{x}$ with confidence level $\alpha$.
The mean predictive interval width ${\rm MPIW}_\alpha$ is the average width of the prediction intervals:
\begin{eqnarray}
{\rm MPIW}_\alpha= \frac{1}{N_{\rm T}} \sum_{i=1}^{N_{\rm T}} \big| \mathcal{I_{\alpha}}(\mathbf{x}_{\rm T}^{(i)}) \big|,
\end{eqnarray}
where $ \mathcal{I_{\alpha}}(\mathbf{x})$ is the prediction interval at point $\mathbf{x}$ with confidence level $\alpha$ and $| \mathcal{I_{\alpha}}(\mathbf{x})|$ the length of the prediction interval $ \mathcal{I_{\alpha}}(\mathbf{x})$.
The prediction uncertainty of the surrogate model is well characterized when ${\rm CP}_\alpha$ is close to $\alpha$. The prediction uncertainty of the surrogate model is small when ${\rm MPIW}_\alpha$ is small.

For the GPBNN method we obtain the interval $\mathcal{I_{\alpha}}(\mathbf{x})$ by sampling $N_v$ realisations of the posterior distribution of the random process $Y_H(\mathbf{x})$ described in \cref{sec::MCGPGH}.
% The expression of $Y_H(\mathbf{x})$ depends on the sampling of the low-fidelity.
% of the noisy BNN: $y_j=\mathit{BNN}_{\boldsymbol{\theta}_j}(\mathbf{x}) + \sigma_j \epsilon_j$
% where $\left(\boldsymbol{\theta}_j,\sigma_j\right)_{j=1}^{N_v}$ is the HMC sample of the posterior distribution of $\left(\boldsymbol{\theta},\sigma\right)$ and the $\epsilon_j$'s are iid Gaussian random variables with mean zero and variance $1$.
The interval $\mathcal{I_\alpha}(\mathbf{x})$ is the smallest interval that contains the fraction $\alpha$ of the realisations $(y_j)_{j=1}^{N_v}$ of $Y_H(\mathbf{x})$.
%The performances for the different methods are compared in \cref{tab::Sampling}.
For the GP 1F model and the AR(1) model the prediction interval is centered at the predictive mean and its half-length is $q_{1-\frac{\alpha}{2}}$ times the predictive standard deviation, where $q_{1-\frac{\alpha}{2}}$ is the $1-\frac{\alpha}{2}$-quantile of the standard Gaussian law, because the posterior distributions are Gaussian.
For the Deep GP model the prediction interval is obtained by Monte-Carlo sampling of the posterior distribution (with $1000$ samples).

We use GP regression with zero mean function and tensorized Mat\'ern ${5}/{2}$ covariance function
for the low-fidelity GP regression.
The implementation we use is from \cite{gpy2014}.
The optimisation for GP regression gives a correlation length of $0.108$.
For this example we choose $N_n=30$ neurons, we use the ReLU function as activation function, and we use the BNN  implementation proposed in \cite{bingham2019pyro}.
The sample size of the posterior distribution of the BNN parameter $(\boldsymbol{\theta},\sigma)$ is $N_v =500$ (see \cref{fig::NvEval}).

\begin{figure}
  \begin{center}
    \hspace{-1.70cm}
    \includegraphics[scale=0.4]{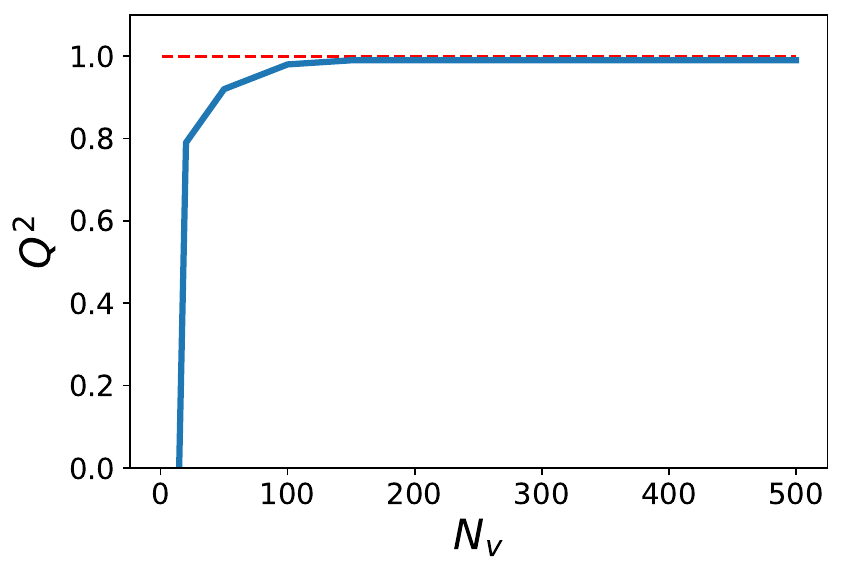} %\hspace{-0.30cm}
    \includegraphics[scale=0.4]{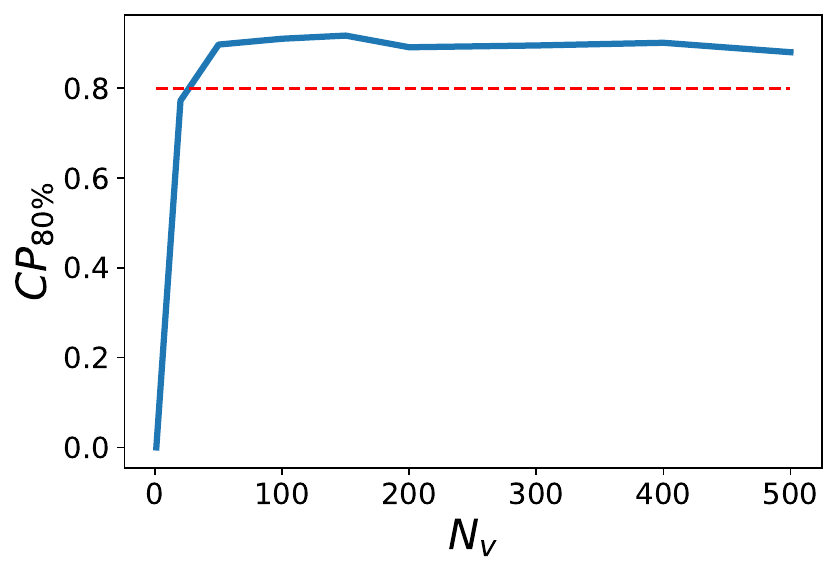} %\hspace{-0.30cm}
    \includegraphics[scale=0.4]{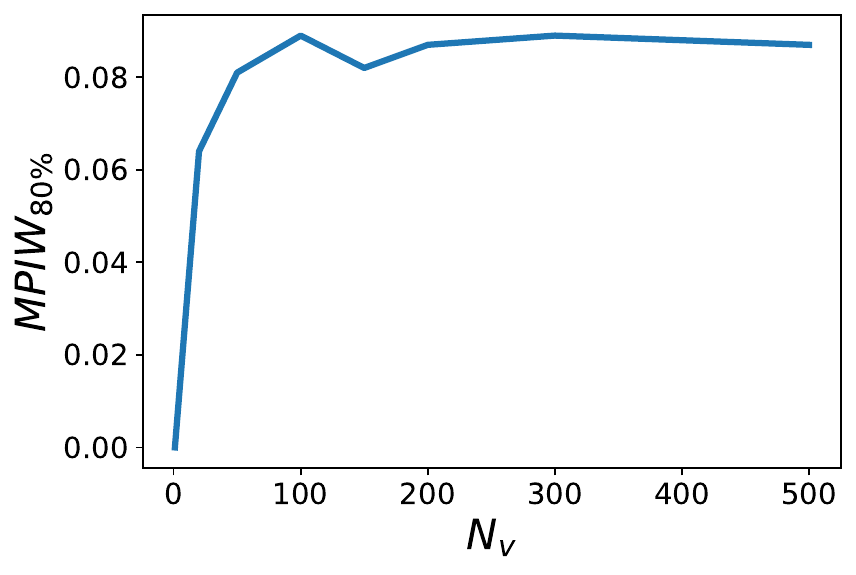}
  \end{center}
  \vspace{-0.25cm}
  \hspace{1.5cm} (a) $Q^2_T $ \hspace{4cm} (b) ${\rm CP}_{80\%}$ \hspace{3.75cm} (c) ${\rm MPIW}_{80\%}$
  \caption{\label{fig::NvEval} Error $Q_{\rm T}^2$, coverage probability at $80\%$ and ${\rm MPIW}_{80\%}$ as functions of $N_v$ for $\bar{I}=\emptyset$.}

\end{figure}

Low-fidelity surrogate models of different accuracies are considered to understand how our strategy behaves under low-fidelity uncertainty.
This is done by considering that low-fidelity data points are only accessible in $[0,1]\backslash \bar{I}$.
We have thus chosen to study three cases, a very good low-fidelity emulator with $\bar{I}=\emptyset$ (for which the $Q^2_{l\rightarrow l}$ of the low-fidelity emulator is $0.99$\textcolor{black}{, with $Q^2_{l\rightarrow l}$ the $Q^2$ of the low-fidelity model for low-fidelity prediction}), a good  emulator with $\bar{I}= [\frac{1}{3},\frac{2}{3}]$ ($Q^2_{l\rightarrow l}=0.98$) and a poor  emulator with $\bar{I}= [\frac{3}{4},1]$ ($Q^2_{l\rightarrow l}=0.84$).
\textcolor{black}{The objective of these different sets is to have more or less accurate low-fidelity models to cover different learning configurations in order to test the model in different configurations.
}

\Cref{tab::Sampling} compares for these examples the different techniques, proposed in \cref{sec::MCGPGH}, for $\bar{I} = \emptyset$.
All methods have the same efficiency in terms of $Q_{\rm T}^2$.
The uncertainties of the predictions are overestimated for all methods.
However, the Gauss-Hermite method has the best ${\rm CP}_\alpha$ and the best uncertainty interval (i.e., the smallest mean predictive interval width ${\rm MPIW}_\alpha$).
The quantiles method and the Mean-Std method also have reasonable ${\rm CP}_\alpha$, but their uncertainty intervals are larger.
%The $80\%$ uncertainty interval is most relevant for the Gauss-Hermite method.
This leads us to use the Gauss-Hermite method.
However, we note that all methods over-estimated the prediction interval.
We believe this is due to the high regularity of the function to be predicted.

\begin{table}
  \caption{Error $Q^2_T$, coverage probability ${\rm CP}_\alpha$  and mean predictive interval width ${\rm MPIW}_\alpha$  for $\alpha=80\%$ and for different methods of sampling. Here $\bar{I}=\emptyset$.}
  \label{tab::Sampling}
  \begin{center}
  \begin{tabular}{c|ccccc}
       & $Q^2_T$     &  ${\rm CP}_\alpha$  &  ${\rm MPIW}_\alpha$   \\
    \hline \\
    Gauss-Hermite $S=5$ & $\mathbf{0.99}$ & $\mathbf{0.88}$ & $\mathbf{0.083}$  \\
    Mean-Std & $\mathbf{0.99}$ & $0.97$ & $0.095$   \\
    Quantiles & $\mathbf{0.99}$  & $0.90$ & $ 0.105$ % \\
    % Random Samples $S=5$ & $ \mathbf{0.99}$ & $0.98 $ & $0.92$ \\
    % Random Samples $S=15$ & $\mathbf{0.99}$ & $0.96$ & $0.090$
  \end{tabular}
  \end{center}
\end{table}

\begin{figure}
  \begin{center}
    \hspace{-0.30cm}
    \includegraphics[scale=0.5]{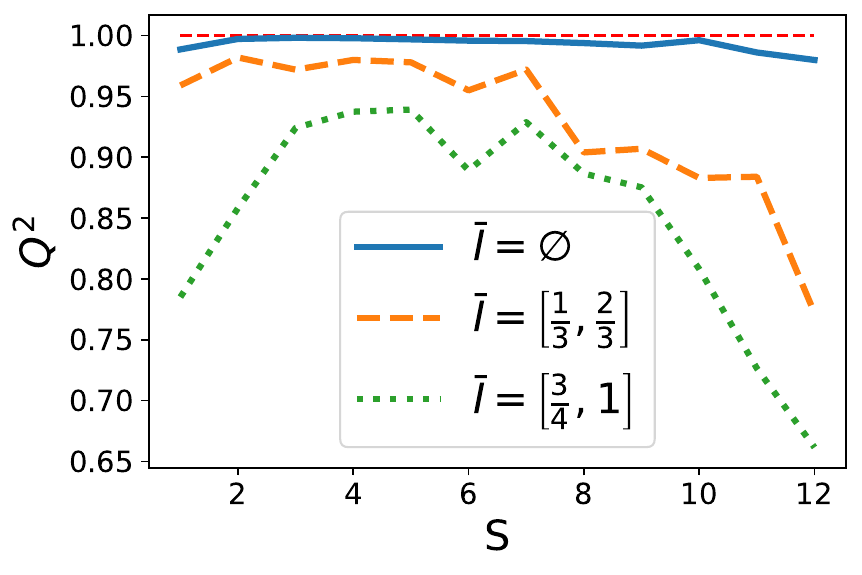} %\hspace{-0.30cm}
    \includegraphics[scale=0.5]{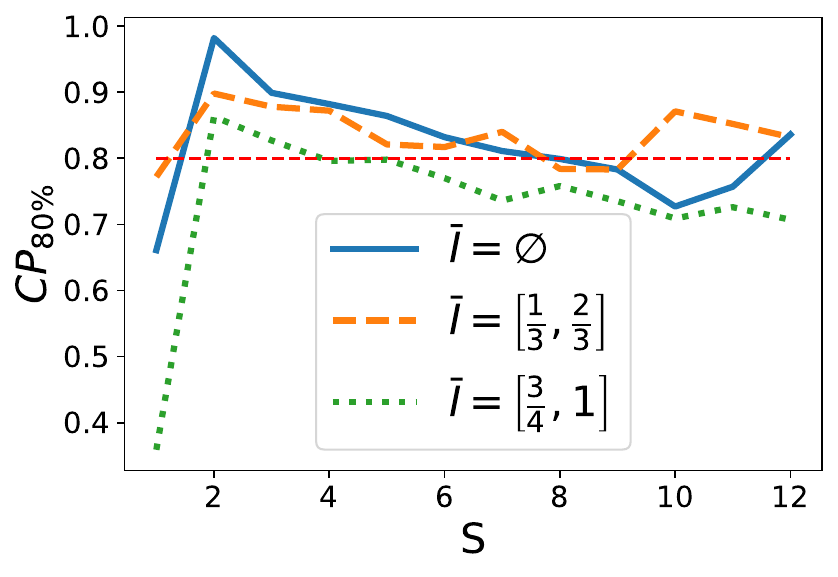} %\hspace{-0.30cm}
    \includegraphics[scale=0.5]{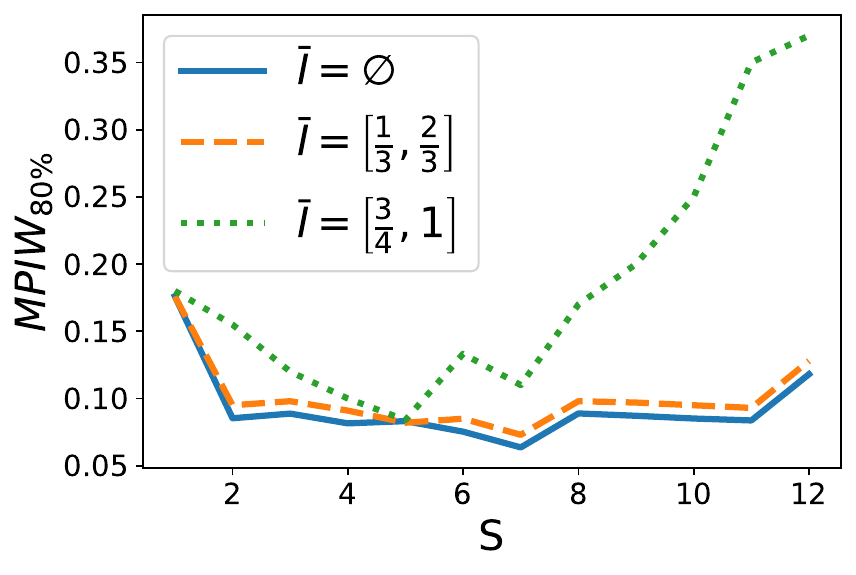}
  \end{center}
  % \vspace{-0.25cm}
  % \hspace{1.5cm} (a) $Q^2_T $ \hspace{4cm} (b) ${\rm CP}_{80\%}$ \hspace{3.75cm} (c) ${\rm MPIW}_{80\%}$
  %\begin{center}
  %  \includegraphics[scale=0.295]{SSizeSTDv3.pdf}\hspace{-0.35cm}
  %  \includegraphics[scale=0.282]{Porbdistv2.pdf}
  %\end{center}
  %\vspace{-0.25cm}
  %\hspace{0.5cm}(c) ${\rm MPIW}_{80\%}$ \hspace{1.65cm} (d) $ \mathit{BNN}(0.2, \tilde{f}_{L}(0.2)) $
  \caption{\label{fig::SEval} Error $Q_{\rm T}^2$, coverage probability at $80\%$ and ${\rm MPIW}_{80\%}$ as functions of $S$.}

\end{figure}

In \cref{fig::SEval} we report the performances of the GPBNN method as functions of $S$ between $1$ and $12$ for different $\bar{I}$.
For $S=1$ the uncertainty is underestimated and the accuracy of the prediction is not optimal, which shows that it is important to exploit the uncertainty predicted by the low-fidelity model.
For $2\leq S\leq 5$ the prediction is good, the error is constant and the $Q_{\rm T}^2$ is maximal as seen in \cref{fig::SEval}(a).
\Cref{fig::SEval}(b) shows that the coverage probability is acceptable for $2\leq S \leq 7$.
Finally, the ${\rm MPIW}_{80\%}$ on \cref{fig::SEval}(c) is minimal for $3 \leq S\leq 7$.
The best value of $S$ is in $\left[2, 5\right]$ depending on the accuracy of the low-fidelity emulator.
\textcolor{black}{
Increasing $S$ should increase the size of the training set.
It is therefore expected that the larger $S$ is, the more accurate the model will be.
But for $S>7$ we have realisations out of the $95\%$ prediction interval.
The linkage between internal BNN non-linearities and realisations outside $95\%$ prediction interval tends to degrade the prediction of the high-fidelity model.
In this case we hypothesize that the BNN learns poorly the low probability values.
}
%In this example the output law is not Gaussian.
%We confirm this because the Kolmogorov-Smirnov test occilate between $0$ and $1$ for $\bar{I}=\emptyset$
%This result was expected because the BNN output is rather a Gaussian mixture.

We have carried out a study on the best value of $N_v$.
We thought that the best value would be the largest possible.
Therefore, we tested for values of $N_v$ ranging from $1$ to $1000$ for $\bar{I}=\emptyset$.
For values of $N_v$ greater than 200 the performance is identical as a function of $N_v$.
For values below $200$ a greater variability was found.
We choose to use $N_v=500$ to have a sufficient margin.
In \Cref{fig::NvEval} we have averaged the estimators for 5 independent training sets.

We now want to compare our GPBNN method with other state-of-the-art methods. 
The single-fidelity GP method used to emulate the high-fidelity code from the $N_H$ high-fidelity points is called GP 1F.
We use the implementation in \cite{gpy2014}.
The multi-fidelity GP regression with autoregressive form introduced by \cite{kennedy2000predicting} and improved by  \cite{le2014recursive} is called autoregressive model AR(1).
The method proposed in both \cite{cutajar2019deep,perdikaris2017nonlinear} is called DeepGP, we use the implementation from \cite{perdikaris2017nonlinear} and the covariance given in \cite{cutajar2019deep} equation (11).
The method from \cite{meng2020multi} is called MBK method.
The MBK method is the combination of a fully connected NN for low-fidelity regression and BNN for high-fidelity.
We implemented it using \cite{bingham2019pyro}.
The  methods that require the minimal assumptions on the function $f_L$ and $f_H$ are the GPBNN and MBK methods.
This is the reason why they are the two methods that are compared  in \cref{fig::quadraex}.

\begin{figure}
  \begin{center}
    \includegraphics[scale=0.5]{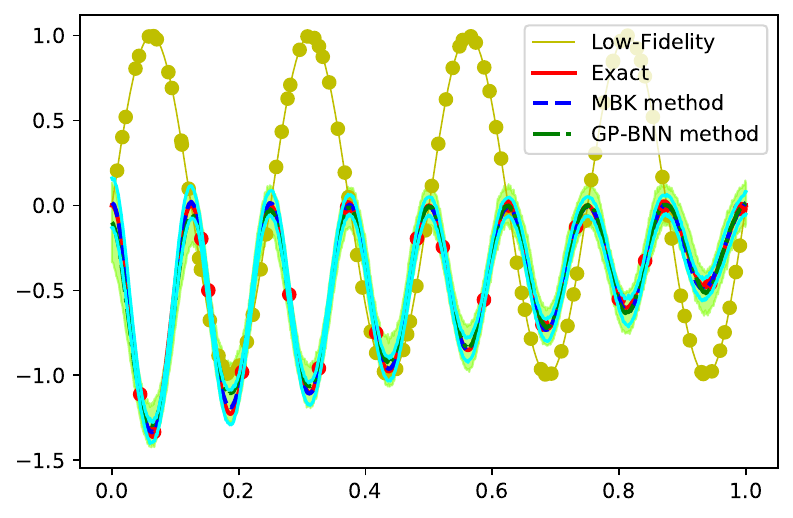}\hspace{-0.2cm}
    \includegraphics[scale=0.5]{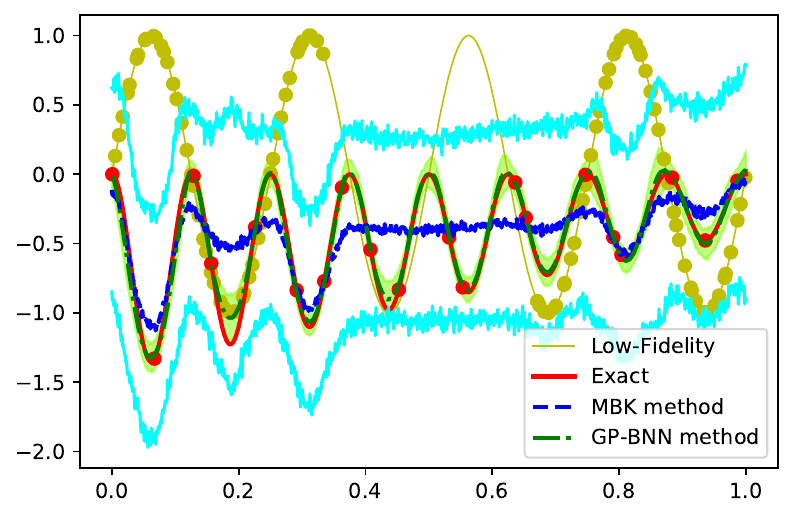}\hspace{-0.2cm}\\
    (a) $\bar{I} = \emptyset $ \hspace{5cm} (b) $\bar{I}=[\frac{1}{3},\frac{2}{3}]$ \hspace{3.75cm} \\
    \includegraphics[scale=0.5]{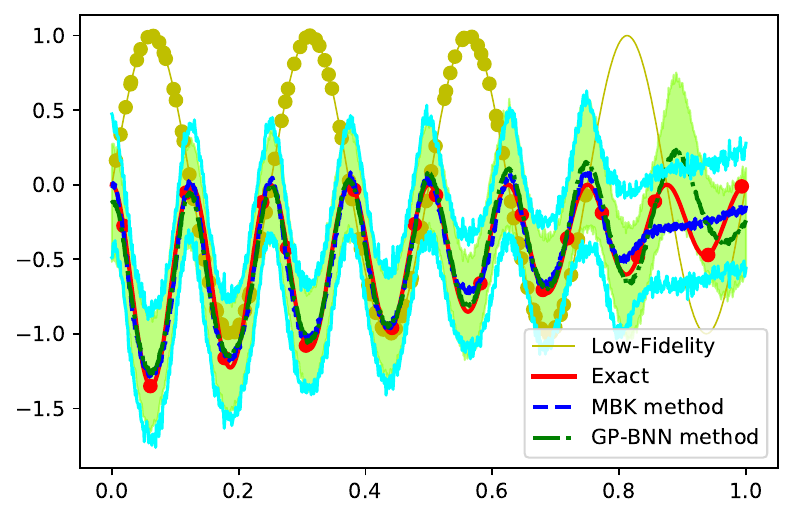}\\
    (c) $\bar{I}=[\frac{3}{4},1]$
    \end{center}
    \vspace{-0.25cm}
    
  % \begin{center}
  % \includegraphics[scale=0.4]{MultiFI1DH20L100S7NN30v3.pdf}\hspace{-0.2cm}
  % \includegraphics[scale=0.4]{MultiFI1DH20L100S7NN30_033v3.pdf}\hspace{-0.2cm}
  % \includegraphics[scale=0.4]{MultiFI1DH20L100S7NN30_075v3.pdf}
  % \end{center}
  % \vspace{-0.25cm}
  % \hspace{1.5cm}(a) $\bar{I} = \emptyset $ \hspace{4cm} (b) $\bar{I}=[\frac{1}{3},\frac{2}{3}]$ \hspace{3.75cm} (c) $\bar{I}=[\frac{3}{4},1]$
\caption{\label{fig::quadraex} Comparison between the MBK method (in blue dashed line) and the GPBNN method with $S=5$ (in green dash-dotted line) for the estimation of the function $f_H$ \cref{eq::highex1Dquad} (in red solid line).
The high-fidelity data points are in red.
The light colored lines (blue and green) plot the predictive intervals.
The uncertainty interval is given by the $\mathcal{I}_{80\%}(\mathbf{x})$ }
\end{figure}

First, the output laws for the MBK method and for the GPBNN presented in \cref{fig::quadraex} are different.
For the MBK method  the posterior law is associated to the high-fidelity BNN knowing both the high-fidelity data and the low-fidelity model.
Unlike the MBK method, the GBBNN method's output law represents the posterior law knowing high-fidelity points and the posterior distribution of the low-fidelity model built from the $N_L$ low-fidelity points.
In \cref{fig::quadraex} $S = 5$ because, as discussed above, this value seems to be the best value when the low-fidelity code is not so accurate.

The presented techniques for multi-fidelity regression are compared in \cref{tab::comptot,tab::CovProb,tab::MPIW1D}.
The GP 1F model and the multi-fidelity AR(1) model do not have good predictive properties.
For the GP 1F model it is due to the lack of high-fidelity data and for the AR(1) model it is due to the strongly nonlinear relationship between the code levels.
The three other methods perform almost perfectly when $\bar{I}=\emptyset$.
The DeepGP is outstanding when $\bar{I}=\emptyset$ but the interaction between the codes (a quadratic form) is given exactly in the covariance structure which is a strong assumption in DeepGP that is hard to verify in practical applications.
When $\bar{I}$ is non-empty the MBK method gives less accurate predictions because the method assumes strong knowledge of the low-fidelity code level.
However, its uncertainty interval seems realistic for this example although too large.
The DeepGP has reasonable errors
but \Cref{tab::CovProb,tab::MPIW1D} show that the predictive uncertainty of the DeepGP method does not fit the actual uncertainty of the prediction (it has poor coverage probability, either too large or too small).
The GPBNN method has the smallest error (best $Q^2_{\rm T}$) and it is predicting its accuracy precisely (it has good and nominal coverage probability; here $S=5$) and the predictive variance and prediction interval width are small compared to the other methods:
The GP 1F and AR(1) models have reasonable coverage probabilities but large mean predictive interval widths.
For this simple illustrative example the GPBNN method seems to be the most suitable method.

\begin{table}[h]
  \caption{$Q_{\rm T}^2$ for different methods and  segments $\bar{I}$ of missing low-fidelity values. Here $S=5$.}\label{tab::comptot}
  \begin{center}
  \begin{tabular}{c|ccccc}
    \small{$\bar{I}$ }    & GP 1F     &  AR(1)  &  DeepGP  &  MBK  &  GPBNN \\
    \hline \\
    \small{$\emptyset$} & $0.12$ & $-0.29$ & $\mathbf{0.99}$ & $\mathbf{0.99}$  & $\mathbf{0.99}$     \\
    \small{$\left[\frac{1}{3},\frac{2}{3}\right]$} & $0.13$ & $-0.34$ & $\mathbf{0.98}$ & $ 0.90$  &  $\mathbf{0.98}$     \\
    \small{$\left[\frac{3}{4},1\right]$} & $0.12$ & $-0.29$ & $0.90$ & $ 0.51$  &  $\mathbf{0.93}$     \\
  \end{tabular}
  \end{center}
\end{table}

\begin{table}
  \caption{Coverage probability ${\rm CP}_\alpha$ for $\alpha=80\%$ and for different methods and  segments $\bar{I}$ of missing low-fidelity values. Here $S=5$.}
  \label{tab::CovProb}
  \begin{center}
  \begin{tabular}{c|ccccc}
    $\bar{I}$   & GP 1F     &  AR(1)  &  DeepGP  &  MBK &  GPBNN \\
    \hline \\
    $\emptyset$ & $\mathbf{0.82}$ & $\mathbf{0.82}$ & $0.99$  & $0.76$ & $\mathbf{0.88}$    \\
    $\left[\frac{1}{3},\frac{2}{3}\right]$ & $\mathbf{0.78}$ & $\mathbf{0.79}$ &  $0.60$  & $ 0.84$ &  $\mathbf{0.83}$   \\
    $\left[\frac{3}{4},1\right]$ & $\mathbf{0.78}$ & $\mathbf{0.82}$ & $0.62$  & $0.86$ & $\mathbf{0.78}$    \\
  \end{tabular}
  \end{center}
\end{table}

\begin{table}
  \caption{Mean predictive interval width ${\rm MPIW}_\alpha$ for $\alpha=80\%$ and for different methods and  segments $\bar{I}$ of missing low-fidelity values. Here $S=5$.}
  \label{tab::MPIW1D}
  \begin{center}
  \begin{tabular}{c|ccccc}
    $\bar{I}$   & GP 1F     &  AR(1)  &  DeepGP  &  MBK &  GPBNN \\
    \hline \\
    $\emptyset$ & $0.44$ & $0.55$ & $\mathbf{0.002}$  & $0.037$ & $\mathbf{0.083}$    \\
    $\left[\frac{1}{3},\frac{2}{3}\right]$ & $0.42$ & $0.45$ &  $\mathbf{0.010}$  & $0.36$ &  $\mathbf{0.082}$   \\
    $\left[\frac{3}{4},1\right]$ & $0.44$ & $0.45$ & $0.097$  & $0.31$ & $\mathbf{0.084}$    \\
  \end{tabular}
  \end{center}
\end{table}

\subsection{2D function approximation}

The CURRIN function is a two-dimensional function, with $\mathbf{x} \in \left[0,1\right]^2$.
This function is commonly used to simulate computer experiments \cite{cutajar2019deep}.
The high- and low-fidelity functions are:
\begin{align}
  f_H(\mathbf{x}) =& \left[ 1 - \exp\left(-\frac{1}{2x_2}\right)\right]
  \frac{2300x_1^3 + 1900x_1^2 + 2092 x_1 + 60}{ 100x_1^3 + 500 x_1^2 + 4x_1 + 20},\\
  f_L(\mathbf{x}) =& \frac{1}{4} \left[f_H(x_1+\delta, x_2 +\delta) + f_H(x_1+\delta, \max\left(0,x_2-\delta\right))\right]\nonumber\\
  &+ \frac{1}{4} \left[f_H(x_1-\delta, x_2 +\delta) + f_H(x_1-\delta, \max\left(0,x_2-\delta\right))\right],
\end{align}
with $\mathbf{x} = [x_1, x_2]$ and $\delta$ the filter parameter.
In  \cite{cutajar2019deep} we have $\delta = 0.05$ and this gives very small differences between the two functions and the prediction of the high-fidelity function by the low-fidelity one has $Q^2_{l\rightarrow h} = 0.98$.
In the following we  set $\delta = 0.1$, and then $Q^2_{l\rightarrow h} = 0.87$.
An additive Gaussian noise is added to the low-fidelity code.
The noise has a zero mean and a variance equal to the empirical variance of the signal $0.08$.

In this example we also consider that the low-fidelity code is costly and we only have a small number of low-fidelity points: $N_L = 25$ and $N_H = 15$.
The high- and low-fidelity points are chosen by maximin Latin Hypercube Sampling (LHS).
The test set is composed of $1000$ independent points following a random uniform law on $[0,1]^2$.

The kernel used for GP regression low-fidelity is a Mat\'ern $5/2$ covariance function.
The predictive error for the GP regressor of the low-fidelity model is $Q^2=0.91$ using a nugget effect in the Gaussian process regression.
The BNN is defined with $N_l = 40$ neurons and the mean and variance are evaluated by \cref{eq::prediction,eq::prediction2} with $N_v = 500$.
$N_l$ could be chosen arbitrarily but the fact that we are in a small data context leads us to choose a small value.
It is possible to use cross-validation to choose $N_l$, but the computer cost would be here prohibitive.
The previous example discussed in \cref{sec::ex1D} suggests choosing $S$ between $3$ and $5$ for the low-fidelity surrogate model sampling.
In this example, due to the large low-fidelity error the model needs a large value of $S$ to account appropriately for the uncertainty and we choose $S= 5$.

All methods have been compared in \cref{tab::Currin}.
The GPBNN model in this example seems to be accurate in terms of $Q_{\rm T}^2$ and in uncertainty quantification.
It is much better than the GP 1F model (single-fidelity GP model built with  the high-fidelity data).
We presume that it is due to the lack of high-fidelity data.
AR(1) model gives better but not satisfying results.
This is expected due to the non-linearity between codes.
The results of the DeepGP method and the MBK method are worse than the one of the GP 1F in error and in uncertainty.
For the DeepGP this can be understood by the fact that the covariance is not well adapted, see \cite{cutajar2019deep}.
And for the MBK the lack of points leads to a very poor optimisation of the hyper-parameters.
%It questions the relevance of these two methods in this context.

\begin{table}
  \caption{Comparison of the multi-fidelity methods on the CURRIN function via $Q_{\rm T}^2$, ${\rm CP}_{80\%}$ and ${\rm MPIW}_{80\%}$.}
  \label{tab::Currin}
  \begin{center}
  \begin{tabular}{c|ccccc}
     & GP 1F     &  AR(1)  &  DeepGP & MBK &  GPBNN\\
    \hline \\
    $Q_{\rm T}^2$ & $0.73$ & $0.80$ & $0.29$ & $0.27$ & $0.88$  \\
     ${\rm CP}_{80\%}$ & $0.68$ & $0.80$ & $0.62$  & $0.57$ &  $0.80$   \\
    ${\rm MPIW}_{80\%}$ & $0.5$ & $1.0$ & $0.13$   & $1.9$ &  $0.51$    \\
  \end{tabular}
  \end{center}
\end{table}

\subsection{Double pendulum}

\begin{figure}
  \centering
  \includegraphics[scale=0.15]{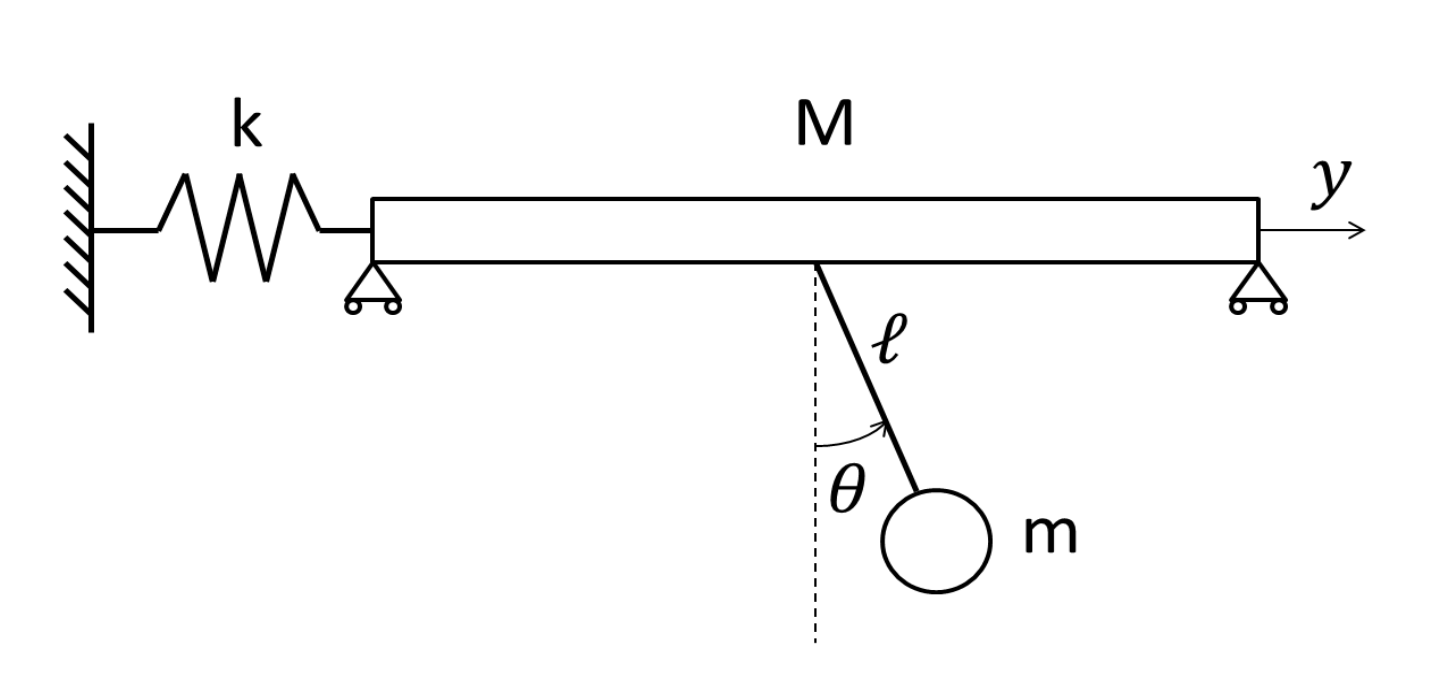}
  \caption{\label{fig::Impend} Illustration of the double pendulum.}
\end{figure}

The double pendulum system can be seen as a dual-oscillator cluster, see \Cref{fig::Impend}.
The system is presented in \cite{perrin2019adaptive}.
%The inputs and there variations are presented in \cite{kerleguer2021multi}.
The inputs of the system are of dimension 5, including $(k, M, \theta, \dot{\theta},y_0)$.
The output is of dimension 1, it is the maximum along the axis $y$ of the position of the mass $m$ in the first $10$ seconds.
% and in \cite{kerleguer2021multi}.
We have two codes: the high-fidelity code numerically solves Newton's equation. The low-fidelity code simplifies the equation, by linearisation for small angles of the pendulum motion, and solves the system.
Our goal is to build a surrogate model of the high-fidelity code using $N_L=100$ and $N_H=20$, with maximin LHS sampling.
The input parameters are  the mass $M \in [10, 12]$, the spring stiffness $k\in[1,1.4]$, the initial angle of the pendulum $\theta_0 \in [\frac{\pi}{4},\frac{\pi}{3}]$, the initial derivative $\dot{\theta_0}\in[0,\frac{1}{10}]$  and the  initial position of the mass $y_0 \in[0,0.2]$.
The fixed parameters are $\dot{y}_0=0$, the gravitational acceleration $g=9.81$, the length of the pendulum $l=2$ and its mass $m=0.5$.
The output is the maximum in time of the amplitude of the mass $m$.
The error is computed on a test set, different for each learning set, of $64$ samples uniformly distributed on the input space.
To evaluate each model we use $5$ independent learning and test sets.

\textcolor{black}{The GP regression on low-fidelity is performed without trend with a Mat\'ern $5/2$ as kernel and the hyper-parameters are estimated with maximum likelihood, using \cite{gpy2014}.}
The $Q^2$ \textcolor{black}{(given in \cref{eq::Q2formula})} for the low-fidelity surrogate model for the GP is $0.98$.
\textcolor{black}{The model is efficient to predict the low-fidelity code output, but the prediction interval is under-evaluated with coverage probability interval of $69\%$ and mean predictive interval width of $0.041$.}
The BNN is defined with \textcolor{black}{the number of neurons} $N_l = 30$.
\textcolor{black}{
The BNN is one hidden layer fully connected with Gaussian priors for weights and bias, as described in \cref{ssec::BNN}.
Therefore, the number of parameters in the BNN is $N_l\times d + 2 \times N_l + 1  = 211$.
We will estimate the posterior distribution of the weights and bias using Hamiltonian Monte Carlo (HMC) algorithm with No-U-Turn Sampler (NUTS).
This relatively small number of parameters allows us to use Pyro's NUTS algorithm (see \cite{bingham2019pyro}) with standard parameters.
The number of estimation samples for mean and prediction interval is $N_v= 500$.
}
\textcolor{black}{The high-fidelity code output is model by GPBNN.
A Gauss-Hermite sample size $S=5$ is chosen according to the results of \cref{sec::ex1D}.}

The results are presented in \Cref{tab::pendulum}.
The prediction of the MBK method is not accurate compared to all the other methods.
We think that this is due to the small data set regarding the dimension of the BNN's input.
However, the uncertainty of prediction is still accurate even if the uncertainty interval is large compared to the other methods (i.e. the coverage probability is close to the target value).
This result is very surprising for us in regard with the poor quality of the low-fidelity model, that has a $Q^2_{l\rightarrow l}$ of $0.7$. 
The DeepGP model shares the best predictive error with GPBNN.
The single fidelity and the AR(1) models display slightly larger errors.
The ${\rm CP}_{80\%}$ values are in the target area for GP1F and AR(1) but they are associated with large prediction intervals.
The DeepGP clearly underestimates the uncertainty of its predictions.
${\rm CP}_{80\%}$ value is good for the GPBNN and close to the target value.
Moreover, the uncertainty interval is the smallest of all methods.
On this real life example the GPBNN is competitive compared to other state-of-the-art methods.

\begin{table}
  \caption{Comparison of the multi-fidelity methods on the pendulum example via $Q_{\rm T}^2$, ${\rm CP}_{80\%}$ and ${\rm MPIW}_{80\%}$.}
  \label{tab::pendulum}
  \begin{center}
  \begin{tabular}{c|ccccc}
     & GP 1F     &  AR(1)  &  DeepGP & MBK &  GPBNN\\
    \hline \\
    $Q_{\rm T}^2$ & $0.93$ & $0.94$ & $\mathbf{0.95}$ & $0.54$ & $\mathbf{0.95}$  \\
     ${\rm CP}_{80\%}$ & $\mathbf{0.81}$ & $\mathbf{0.78}$ & $0.62$  & $0.88$ &  $\mathbf{0.80}$   \\
    ${\rm MPIW}_{80\%}$ & $0.154$ & $0.146$ & $0.069$   & $0.859$ &  $\mathbf{0.101}$    \\
  \end{tabular}
  \end{center}
\end{table}

\section{Conclusion}
\label{sec::Disc}
%This work shows an opportunity for neural networks in the field of uncertainty quantification.
%Thanks to the multi-fidelity context it is possible to process small data sets.
Our main focus in this paper is to give the Gaussian Process regression posterior distribution of a low-fidelity model as input to a Bayesian neural network for multi-fidelity regression.
Different methods are proposed and studied to transfer the uncertain predictions of the low-fidelity emulator to the high-fidelity one, which is crucial to obtain minimal predictive errors and accurate predictive uncertainty quantification.
The Gauss-Hermite quadrature method is shown to significantly improve the predictive properties of the BNN.
The conducted experiments show that the GPBNN method is able to process noisy and real life problems.
Moreover, the comparison with some state-of-the-art methods for multi-fidelity surrogate model highlight the precision in prediction and in uncertainty quantification.
%In the future, we would like to investigate the issue of high dimensionality, both in input and output, as well as active learning.
%The BNN needs more investigations in order to evaluate the hypothesis on the regularity of the output function.

It is possible to extend the GPBNN method to a hierarchical multi-fidelity framework with more than two levels of codes sorted by increasing accuracy.
The natural approach is to consider the output of a low-fidelity metamodel as an input to the next metamodel.
The problem that arises from this naive extension is that the input sample size increases from one metamodel to the other one. 
To solve this problem, two different methods can be considered.
For simplicity, let us consider the case of three code levels.
The first method is to consider that the metamodel for the lowest fidelity level carries no uncertainty.
The low-fidelity metamodel prediction is then added to the low fidelity inputs of a GPBNN modeling the two highest fidelity code levels.
The second method consists in approximating the mid-fidelity output uncertainty of the BNN as Gaussian.
This means modeling the system by one GP for the low-fidelity level and two different BNNs for the other fidelity levels.
It becomes possible, with the Gaussian approximation, to use less expensive sampling methods such as the Gauss-Hermite quadrature.
Of course other methods could be considered to control the sample size.
We also anticipate that the GPBNN method could be extended beyond the hierarchical case of a sequence of simulators ranked from lowest to highest fidelity.  
Indeed, it should be possible to deal with a general Markov case in which the different fidelity levels are connected via a directed acyclic graph \cite{ji2021graphical}.

The interest of combining regression method for multi-fidelity surrogate modelling is not to be proved, but this paper adds the heterogeneity of models for multi-fidelity modelling.
To be able to combine heterogeneous models into one model and to consider the uncertainty between them is one of the keys to adapt the multi-fidelity surrogate model to a real-life regression problem.
%We imagine that the GP regression of the low-fidelity code already exist and with the development of the high-fidelity code came the need of a surrogate model.
%The GPBNN will only need BNN leaning if a more version of the code is developed.

The number of elements in the learning set is not a problem any more thanks to \cite{rulliere2018nested}.
% We have been able to use many points in the learning set.
This approach could be used for the GP part of the GPBNN.
With more time in the training the BNN part will be able to deal with many points, even if this ability is less critical because $N_L \gg N_H$.
Consequently, the GPBNN can be extended in order to tackle larger data sets.

Existing autoregressive models and Deep GP can only be used for low-dimensional outputs.
We wish to extend the method to high-dimensional outputs.
Dimension reduction techniques have already been used as principal component analysis or autoencoder, as well as tensorized covariance methods \cite{perrin2019adaptive} that remain to be extended to the multi-fidelity context.
However, neural networks are known to adapt to high-dimensional outputs.
We should further investigate how to build multi-fidelity surrogate models with functional input/output in the context of small data.
The ability to construct models that are tractable for high-dimensional input and output is key for further research.

%% The Acknowledgements part is started with the command \acknowledgements;
%% acknowledgements are then done as normal sections before appendix
%% \acknowledgements

%% \acknowledgements

%% This research was supported by the Computational Mathematics program of
%% AFOSR (grant FA9550-07-1-0139).

%% appendix sections are then done as normal sections and after Acknowledgements



%\begin{thebibliography}{-8}


%\small{
%\bibitem{key}

%...

%}

%\end{thebibliography}


\bibliographystyle{IJ4UQ_Bibliography_Style}

\bibliography{references}
\end{document}